\documentclass[a4paper,dvipsnames]{article}

\usepackage{amsthm,amssymb,amsmath,enumerate,graphicx,epsf,bm,caption,framed}
\usepackage{authblk}
\usepackage[greek,english]{babel} 
\usepackage[percent]{overpic}
\usepackage{epstopdf}

\newcommand{\COLORON}{0}
\newcommand{\NOTESON}{0}
\newcommand{\Debug}{0}
\usepackage[usenames]{color} 
\usepackage{amsthm,amssymb,amsmath,bbm,enumerate,graphicx,epsf,stmaryrd}
\usepackage[bookmarks, colorlinks=false, breaklinks=true]{hyperref} 

\newcommand{\comment}[1]{}
\newcommand{\COMMENT}[1]{}

\definecolor{darkgray}{rgb}{0.3,0.3,0.3}
\newcommand{\defi}[1]{{\color{darkgray}\emph{#1}}}

\comment{
	\begin{lemma}\label{}	
\end{lemma}
\begin{proof}

\end{proof}

\begin{theorem}\label{}
\end{theorem} 
\begin{proof} 	

\end{proof}

}

\newtheorem{proposition}{Proposition}[section]
\newtheorem{definition}[proposition]{Definition}
\newtheorem{theorem}[proposition]{Theorem}
\newtheorem{corollary}[proposition]{Corollary}

\newtheorem{lemma}[proposition]{Lemma}

\newtheorem{examp}[proposition]{Example}

\newcommand{\FIG}{0}

\ifnum \NOTESON = 1 \newcommand{\note}[1]{ 

\hspace*{-30pt}
	{\color{blue}  NOTE: \color{Turquoise}{\small  \tt \begin{minipage}[c]{1.1\textwidth}  #1 \end{minipage} \ignorespacesafterend }} 
	
	}
\else \newcommand{\note}[1]{} \fi

\newcommand{\afsubm}[1]{ \ifnum \Debug = 1 {\mymargin{#1}}
\fi} 

\ifnum \Debug = 1 
\else  \fi

\ifnum \FIG = 1 \newcommand{\fig}[1]{Figure ``{#1}''}
\else \newcommand{\fig}[1]{Figure~\ref{#1}} \fi

\ifnum \FIG = 1 
\else  \fi

\ifnum \Debug = 1 \usepackage[notref,notcite]{showkeys}
\fi

\ifnum \COLORON = 0 \renewcommand{\color}[1]{}
\fi

\newcommand{\N}{\ensuremath{\mathbb N}}
\newcommand{\R}{\ensuremath{\mathbb R}}

\newcommand{\Z}{\ensuremath{\mathbb Z}}

\newcommand{\cc}{\ensuremath{\mathcal C}}

\newcommand{\ce}{\ensuremath{\mathcal E}}

\newcommand{\cp}{\ensuremath{\mathcal P}}

\newcommand{\cs}{\ensuremath{\mathcal S}}

\newcommand{\oo}{\ensuremath{\omega}}

\makeatletter
\DeclareRobustCommand{\cev}[1]{%
  \mathpalette\do@cev{#1}%
}
\newcommand{\do@cev}[2]{%
  \fix@cev{#1}{+}%
  \reflectbox{$\m@th#1\vec{\reflectbox{$\fix@cev{#1}{-}\m@th#1#2\fix@cev{#1}{+}$}}$}%
  \fix@cev{#1}{-}%
}
\newcommand{\fix@cev}[2]{%
  \ifx#1\displaystyle
    \mkern#23mu
  \else
    \ifx#1\textstyle
      \mkern#23mu
    \else
      \ifx#1\scriptstyle
        \mkern#22mu
      \else
        \mkern#22mu
      \fi
    \fi
  \fi
}

\makeatother

\newcommand{\g}{\ensuremath{G\ }}
\newcommand{\G}{\ensuremath{G}}

\newcommand{\Ex}{\mathbb E}
\renewcommand{\Pr}{\mathbb{P}}

\newcommand{\Lr}[1]{Lemma~\ref{#1}}

\newcommand{\Tr}[1]{Theorem~\ref{#1}}
\newcommand{\Trs}[1]{Theorems~\ref{#1}}
\newcommand{\Sr}[1]{Section~\ref{#1}}

\newcommand{\Prr}[1]{Pro\-position~\ref{#1}}

\newcommand{\Cr}[1]{Corollary~\ref{#1}}

\newcommand{\Dr}[1]{De\-fi\-nition~\ref{#1}}

\newcommand{\Cg}{Cayley graph}

\newcommand{\fe}{for every}

\newcommand{\st}{such that}

\newcommand{\ti}{there is}

\newcommand{\wrt}{with respect to}

\newcommand{\labtequ}[2]{
 \begin{equation} \label{#1} 	\begin{minipage}[c]{0.9\textwidth}  #2 \end{minipage} \ignorespacesafterend \end{equation} }

\newcommand{\mymargin}[1]{
 \ifnum \Debug = 1
  \marginpar{%
    \begin{minipage}{\marginparwidth}\small%
      \begin{flushleft}%
        {\color{blue}#1}%
      \end{flushleft}%
   \end{minipage}%
  }%
 \fi
}%

\newcommand{\mySection}[2]{}

\usepackage{xcolor}

\usepackage{mathtools}

\RequirePackage{latexsym} \RequirePackage{amsthm}
\RequirePackage{amsmath}
\usepackage{hyperref}

\newtheorem{question}[proposition]{Question}

\usepackage{enumerate}
\RequirePackage{amssymb}\RequirePackage{makeidx}
\usepackage{dsfont}
\usepackage{mathrsfs}

\newtheorem{remark}[proposition]{Remark}

\newcommand{\myremark}[1]{\ifnum \Debug = 1 \tiny #1 \fi}

\newcommand{\pint}{interface}
\newcommand{\mpint}{multi-interface}

\newcommand{\ar}[1]{\vec{#1}}
\newcommand{\ard}[2]{\ensuremath{\vec{#1^{#2}}}}
\newcommand{\ra}[1]{\cev{#1}}

\newcommand{\qtl}{quasi-transitive planar lattice}

\newcommand{\MS}{\mathcal {MS}}

\newcommand{\pcs}{\ensuremath{\dot{p}_c}}

\newcommand{\ipint}{inner-\pint}

\usepackage{authblk}

\title{On the exponential growth rates of lattice animals and interfaces}

\author[1]{Agelos Georgakopoulos}
\author[2]{Christoforos Panagiotis}
\affil[1,2]{{Mathematics Institute}\\
{University of Warwick}\\
{CV4 7AL, UK}\thanks{Supported by the European Research Council (ERC) under the European Union's Horizon 2020 research and innovation programme (grant agreement No 639046).}\\}
\affil[2]{{Universit{\'e} de Gen{\`e}ve}\\
{Section de Math{\'e}matiques}\\
{rue du Conseil-G{\'e}n{\'e}ral 7-9}\\
{1205 Geneva, Switzerland}}

\begin{document}
\date{}
\maketitle

\begin{abstract}
We introduce a formula for translating any upper bound on the percolation threshold  of a lattice \g into a lower bound on the exponential growth rate of lattice animals $a(G)$ and vice-versa. We exploit this to improve on the best known asymptotic bounds on $a(\Z^d)$ as $d\to \infty$. Our formula remains valid if instead of lattice animals we enumerate certain sub-species called interfaces. Enumerating interfaces leads to functional duality formulas that are tightly connected to percolation and are not valid for lattice animals, as well as to strict inequalities for the percolation threshold. 


Incidentally, we prove that the rate of the exponential decay of the cluster size distribution of Bernoulli percolation is a continuous function of $p\in (0,1)$.
\end{abstract}

\section{Introduction}
We improve on the best known asymptotic bounds on the exponential growth rates $a(\Z^d)$ of lattice animals as $d\to \infty$, using a probabilistic method involving percolation theory. We begin this quest by setting up the machinery in this paper, and complete it in the follow-up work \cite{ExpGrowth2}. Along the way we also obtain new results on percolation.
\subsection{Lattice animals}
A \defi{lattice animal} is a connected subgraph $S$ of the hypercubic lattice $\Z^d$, or more generally, of a vertex-transitive graph. The counts of lattice animals with prescribed parameters have been extensively studied by scholars in statistical mechanics as well as  combinatorics and computer science \cite{SiteAnimals,BareShal,DelTai,GP,HammondExpRates,Harr82, LatticeTrees,LatticeTreesTerms, PG95}, both in $\Z^d$ and  other lattices \cite{BaRoShaImp,BaShaZheImp,RanWelAni}. A lot of the motivation comes from the study of random configurations in $\Z^d$, the central theme in many models of statistical mechanics. Our focus is Bernoulli (bond or site) percolation. Given a value of the percolation parameter $p\in [0,1]$, we can express the probability $\theta(p)$ that the cluster $C_o$ of the origin $o$ is infinite as
\labtequ{eq LA}{$1-\theta(p) = \sum_{A \text{ is a lattice animal containing } o} \Pr_p(A \text{ occurs})$}
The probability $\Pr_p(A \text{ occurs})$ is easily expressed as $p^{|A|} (1-p)^{|\partial A|}$ by the definition of the model, where the \defi{boundary} $\partial A$ comprises the edges outside $A$ that are incident with $A$. Thus if we could enumerate the set of lattice animals $A_{n,m}$ with size $n$ and boundary size $m$ accurately enough, we could answer any question of percolation theory such as e.g.\ the continuity of $\theta(p)$ or the exact value of the critical threshold $p_c$. In practice, however, $A_{n,m}$ is too difficult to enumerate, and one studies asymptotics as $n,m\to \infty$. Such asymptotics, specifically the exponential growth rates, are still informative enough about the behaviour of percolation, and the present paper provides results in this vain.
\subsection{Interfaces}
Certain subfamilies of lattice animals, called interfaces, have also been extensively studied either as a tool in the study of statistical mechanics or for their own sake \cite{CerWul,DoKoShlWul}. The term \defi{interface} is commonly used to denote the common boundary of two components of a crystal or liquid that are in a different phase. The precise meaning of the term varies according to the model in question and the perspective of its study. When studying percolation on a planar lattice $L$ for example, the interface of the cluster $C_o$ of the origin can be thought of as the minimal cut separating $C_o$ from infinity, which forms a cycle in the dual $L^*$. Peierls' argument \cite{Grimmett} is a famous application of the notion that uses an upper bound on the number of such cycles to deduce an upper bound on $p_c(L)$.

In \cite{analyticity} we introduced a variant of the notion of interface, and used it to prove the analyticity of the percolation density $\theta(p)$ for supercritical Bernoulli percolation on $\Z^d$. Our proof relied on the fact that our interfaces satisfy a generalisation of \eqref{eq LA}. Therefore, as argued above for lattice animals, counting interfaces accurately enough would yield important conclusions in percolation theory. This observation is the main motivation of this paper, which studies the exponential growth rate \defi{$b(G)$} of interfaces of a `lattice' \G. Our first result is the following inequality relating $b(G)$ to ${p}_c(G)$:
\labtequ{b ge frp}{$\boxed{b(G) \geq f(r({p}_c(G)))}$,}
where $f(r):= \frac{(1+r)^{1+r}}{r^r}$ and $r(p):=\frac{1-p}{p}$ are universal  functions.  The class $\cs$ of \defi{lattices} \g we work with includes the standard cubic lattice in $\Z_d, d\geq 2$, as well as all quasi-transitive planar lattices (see \Sr{pints site-pints} for definitions).

The fact that interfaces are a sub-species of lattice animals leads to the inequality $a(G)\geq b(G)$. Inequality~\eqref{b ge frp} allows us to translate any upper bound on the percolation threshold  of a lattice \g into a lower bound on the exponential growth rate of lattice animals $a(G)$ and vice-versa. We exploit this in both directions in the follow-up paper \cite{ExpGrowth2}. We improve on the best known asymptotic lower and upper bounds on $a(\Z^d)$ as $d\to \infty$, answering a question of Barequet et al.\ \cite{SiteAnimals}. We use percolation as a tool to obtain the latter, and conversely we use the former to obtain lower bounds on $p_c(\Z^d)$. 
\subsection{Lattice animals vs.\ interfaces}

Inequality \eqref{b ge frp} remains true if we replace $b(G)$ by the exponential growth rate \defi{$a(G)$} of lattice animals, and all ingredients for its proof are available in \cite{HammondExpRates}. In this paper we provide a unified, and simpler, approach to proving such inequalities. Our definition of interface is parametrised by a choice of a basis \cp\ of the cycle space $\cc(G)$ of the lattice \G. We elaborate on this notion in Sections \ref{pints site-pints} and \ref{Growth rates}. By varying the choice of \cp\ we obtain a spectrum of notions of \cp-interface, which are always relevant to percolation: we showed in \cite[Theorem 10.4]{analyticity} (restated as \Tr{unique} below) that every finite connected subgraph of \G, e.g.\ a percolation cluster, contains a unique \cp-interface. One extreme of this spectrum is where \cp\ is just the set of all cycles of \G, in which case the \cp-interfaces coincide with the lattice animals. The other extreme is where \cp\ is a minimal basis, in which case the \cp-interface of a percolation cluster is a thin layer incident with its boundary to infinity. For example, when \g is the square lattice, and \cp\ comprises its squares of length $4$, then the (unique) \cp-interface contained in a cluster $C$ consists of the vertices and edges bounding the unbounded face $F$ of $C$ in the plane. 
Some examples of   such interfaces are depicted in \fig{figscv}. The main message of this paper is that this second extreme of  \cp-interfaces provides finer information about percolation. In particular, we will prove (\Tr{strict}) that for every basis \cp\ comprising cycles of bounded length the following holds:
\begin{theorem} \label{strict intro}
$b(G)< a(G)$ holds \fe\ $G\in \cs$. 
\end{theorem}

We remark that using this theorem and inequality~\eqref{b ge frp} we obtain the strict inequality $${a(G) > f(r({p}_c(G)))}.$$

\subsection{Growth rates parametrized by `volume-to-surface\\ ratio'}

Before stating our other results we need to introduce more terminology. Recall that when considering counts $A_{n,m}$ of lattice animals  it was important to parametrize them both by their size $n$ and their boundary size $m$. Alternatively, instead of $m$ we could use the `volume-to-surface ratio' $n/m$.
Similarly, to establish \eqref{b ge frp}, we consider the exponential growth rate $b_r=b_r(G)$ of the number of
{interfaces} of $G$ with size $n$ and volume-to-surface ratio approximating $r$, as a function of $r\in \R_+$. By `volume' here we mean the number of edges contained in an interface $P$, and by `surface' we mean the cardinality of a set $\partial P$ of edges that are incident with $P$ and are `accessible' from infinity. (See \Dr{pint} for details.) In the example of \fig{figscv}, the edges in $\partial P$ are depicted by dashed lines.

We consider this function $b_r(G)$ to be of independent interest; in fact, most of this paper revolves around it. In particular, we prove that $b_r(G)$ is always continuous (\Tr{cont}) and log-concave (\Tr{concave}).

\medskip

One of the best known results of percolation theory is the exponential decay, as $n\to \infty$, of the cluster size distribution $\Pr_p(|C_o|= n)$ for $p$ in the subcritical interval $[0,p_c)$ \cite{AB}. In the supercritical case $p\in (p_c,1)$ this exponential decay holds for some, but not all, lattices and values of $p$ \cite{AiDeSoLow, HerHutSup}. 

Incidentally, we prove in the Appendix that the rate of the exponential decay of $\Pr_p(|C_o|= n)$, defined as $c(p):= \lim_n \left(\Pr_p(|C_o|= n)\right)^{1/n}$, is a continuous function of $p$ (\Tr{cp cont}). Our proof boils down to elementary calculations not involving our notion of interface. 
Another contribution of this paper is the subexponential decay at $1-p_c$ for triangulated lattices in $\R^d$  (\Cr{cor no ed}). 

Letting $S_o \subseteq C_o$ denote the interface of $C_o$, we can analogously ask for which $p\in (0,1)$ we have  exponential decay of the probability $\Pr_p(|S_o|= n)$. We prove that 
this is uniquely determined by the value $b_{r(p)}$, where $r(p):=\frac{1-p}{p}$ is a bijection between the parameter spaces of edge density $p$ and volume-to-surface ratio $r$.  More concretely, we observe that, firstly, $b_{r(p)}(G) \leq f(r(p))$ holds \fe\ lattice \g and every $p\in (0,1)$, where $f(r)$ is the aforementioned universal function (\Prr{br ineq}), and secondly, $\Pr_p(|S_o|= n)$ decays exponentially in $n$ for exactly those values of $p$ for which this inequality is strict:
\begin{theorem}\label{S_o}
Let $G\in \mathcal{S}$. Then for every $p\in (0,1)$, the interface size distribution $\Pr_p( |S_o|=n )$ fails to decay exponentially in $n$ if and only if $$b_{r(p)}(G)= f(r(p)).$$ 
\end{theorem}

We expect our results to hold for all vertex-transitive 1-ended graphs (this is so for \Tr{S_o}), but decided to restrict our attention to $\cs$ to avoid technicalities that would add little to the understanding of the matter.


It is interesting that \Tr{S_o} holds for every choice of basis \cp\ \wrt\ which our interfaces are defined. In particular, since lattice animals are a special case of \cp-interfaces as mentioned above, we can replace the interface size distribution $\Pr_p( |S_o|=n )$ by the cluster size distribution $\Pr_p( |C_o|=n )$ and $b_{r(p)}(G)$ by its analogue $a_{r(p)}(G)$ counting lattice animals. This form of \Tr{S_o} was proved by Hammond \cite{HammondExpRates} building on a result of Delyon \cite{DelTai}. 

Our proof of \Tr{S_o} is based on a large deviation principle for interfaces (\Lr{large deviation}) that may be of independent interest. Roughly speaking, the latter result says that for any $p\in (0,1)$, most occurying \pint s have a volume-to-surface ratio close to the value $r(p)$.

\medskip

We mentioned above that the most refined extreme of such results is obtained when \cp\ is a minimal basis. This is even more so for `triangulated' lattices, i.e.\ lattices having a basis \cp\ consisting of triangles. For such a lattice, we show that 

\begin{theorem} \label{dual br}
We have $b_r= (b_{1/r})^r$ for every $r>0$.
\end{theorem}
In other words, the values of $b_r$ for $r<1$ determine those for $r>1$ (\Trs{dual site} and \ref{dual bond}). This is the technically most involved result of this paper. It shows that considering interfaces rather than lattice animals yields a more interesting function $b_r$, namely one with a smaller intersection with $f(r)$.
Most of our knowledge about $b_r$ is summarized in \fig{b_r vs f(r)}.


Amusingly, our universal function $f(r)$ also satisfies the equation of \Tr{dual br}, i.e.\ $f(r)= f(1/r)^r$.

\subsection{Improvements for non-amenable graphs}

A well-known theorem of Benjamini \& Schramm \cite{BeSchrPer} states that $p_c(G)\leq\frac{1}{h(G)+1}$, where $h(G)$ denotes the Cheeger constant. We show that this inequality is in fact strict, i.e.\ $$p_c(G)<\frac{1}{h(G)+1},$$ when $G$ is $1$-ended, has bounded degrees, and its cycle space admits a basis consisting of cycles of bounded length (\Tr{strict BS}). Moreover, in \Sr{analogue} we define a variant $I(G)$ of the Cheeger constant by considering interfaces rather than arbitrary finite subgraphs of $G$. We obtain the strengthening $p_c(G)\leq\frac{1}{I(G)+1}$ of the aforementioned theorem (\Tr{cheeger-like}), which again has a site and a bond version. We remark that, unlike $h(G)$, our $I(G)$ can be positive even for amenable graphs. When \g is the planar square lattice for example, it is not hard to see that $I(G)=1/2$ in the bond case, which yields the Peierls bound $p_c\leq 2/3$. Moreover, one can  have $I(G)>h(G)$ even in the non-amenable case: this turns out to be the case for regular triangulations and quadrangulations of the hyperbolic plane as proved in a companion paper \cite{SitePercoPlane}.

\begin{figure}[!ht]
\begin{framed}
\vspace*{5mm}
\centering
\noindent
\begin{overpic}[width=.6\linewidth]{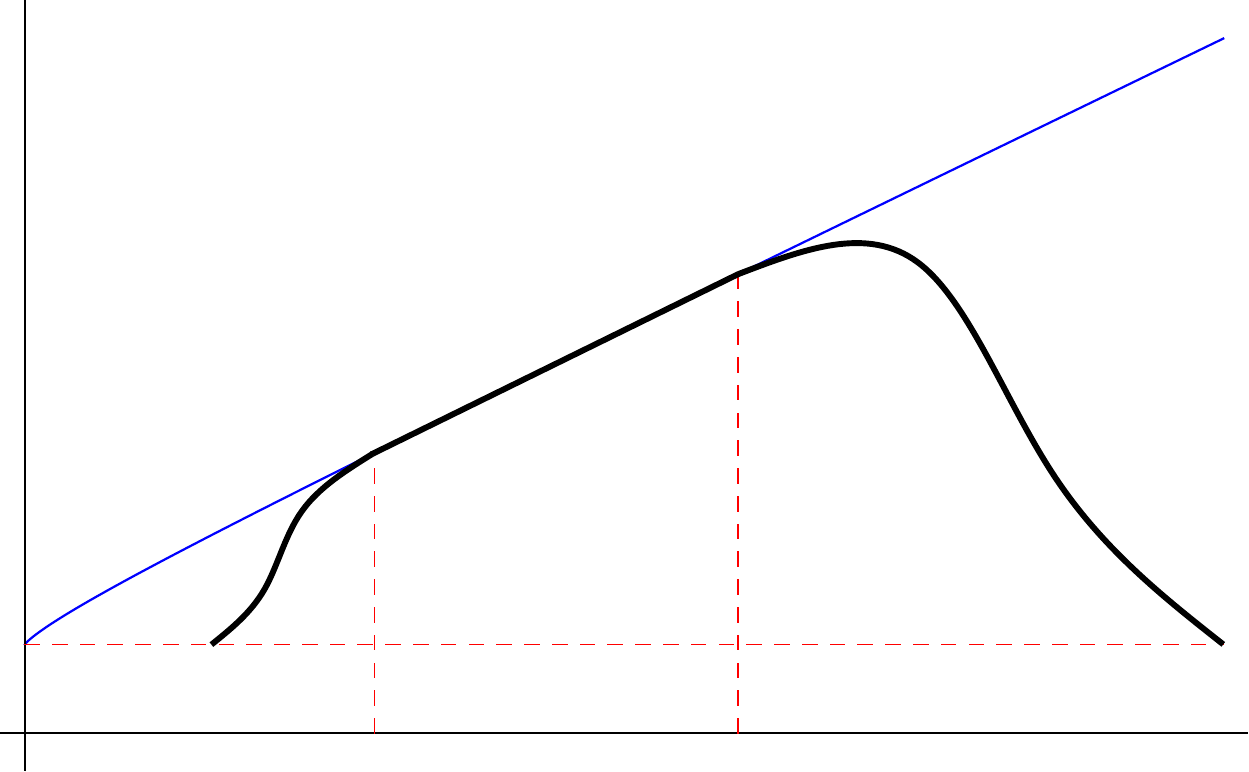}
\put(94,18){$b_r$}
\put(98,-1){$r$}
\put(54,-1){$r(p_c)$}
\put(21,-1){$r(1-p_c)$}
\put(94,48){$f(r)$}
\end{overpic}
\begin{minipage}[c]{0,95\textwidth}
\vspace*{8mm}
\caption{ \small An approximate, conjectural, visualisation of $b_r(G)$ when $G$ is a lattice in $\R^d, d\geq 3$. The graph of $b_r(G)$ (depicted in bold ink) lies below the graph of $f(r):= \frac{(1+r)^{1+r}}{r^r}$ (depicted in blue, if colour is shown). The fact that $f(r)$ plots 
almost like a straight line can be seen by rewriting it as $(1+r)(1+1/r)^r$. The fact that $b_r= f(r)$ for $r$ in the interval $(r(1-p_c), r(p_c)]$, where $r(p):= \frac{1-p}{p}$, follows by combining a theorem of Kesten \& Zhang \cite{KeZhaPro}, saying that exponential decay of $\Pr_p(|S_o|= n)$ fails in that interval, with our \Tr{S_o}. That $b_r< f(r)$ for $r>r(p_c)$ follows from the well-known exponential decay of $\Pr_p(|C_o|= n)$ for $p<p_c$ \cite{AB}. \vspace*{2mm} \\
We also know that $b_r$ is continuous and log-concave. The continuity of $b_r$, combined with \Tr{S_o} again, implies failure of the exponential decay at $p=1-p_c$ (\Cr{cor no ed}), which was not obtained in \cite{KeZhaPro}. \vspace*{2mm} \\ 
If the cycle space of $G$ is generated by its triangles, then \Tr{dual br} determines the subcritical branch $r>r(p_c)$ given the branch $r< r(1-p_c)$ and vice-versa. For the planar triangular lattice the picture degenerates as $p_c=1-p_c=1/2$, and so $b_r= f(r)$ for $r=r(1/2)=1$ only. \vspace*{2mm} \\
Note that $b_r(G)$ is an invariant of \g defined without reference to any random experiment. The connection to percolation is established by \Tr{S_o} via the above transformation $r(p)$. \vspace*{2mm} \\
Since $r(p)$ is monotone decreasing in $p$, the right hand side of \fig{b_r vs f(r)} corresponds to the subcritical percolation regime, and the left hand side to the supercritical. Using the transformation $r\to \frac1{r}$ (from volume-to-surface into surface-to-volume ratio) we could reverse the picture to have the `subcritical' interval on the left. For `triangulated' lattices the picture would look exactly the same due to \Tr{dual br}, only the positions of $r(p_c)$ and $r(1-p_c)$ would be interchanged.}

\label{b_r vs f(r)}
\end{minipage}
\end{framed}

\end{figure}


\subsection{Structure}
This paper is structured as follows. \Sr{SecDef} contains some standard definitions and results that will be used in several places later on. \Sr{pints site-pints} recalls the notion of an interface and some of its properties obtained in \cite{analyticity}. \Sr{Growth rates} introduces $b_r$ and presents the proof of the inequality $b_r\leq f(r)$ and \Tr{S_o}. Sections~\ref{Duality} and \ref{cont dec} contain a proof of \Tr{dual br} and the fact that $b_r$ is well-behaved, namely continuous and log-concave. The last two sections are devoted to the new bounds on $p_c$ for non-amenable graphs and to the proof of \Tr{strict intro}. The Appendix contains a proof that the rate of the exponential decay of the cluster size distribution of Bernoulli percolation is a continuous function of $p\in (0,1)$. An analogous result is proved for interfaces as well.

\section{Definitions and preliminaries}\label{SecDef}

\subsection{Percolation}

We recall some standard definitions of percolation theory. For more details the reader can consult e.g.\ \cite{Grimmett,LyonsBook}. 

Let $G = (V,E)$ be a countably infinite graph, and let $\Omega:= \{0,1\}^E$ be the set of \defi{percolation instances} on \G. We say that an edge $e$ is \defi{vacant} (respectively, \defi{occupied}) in a percolation instance $\oo\in \Omega$, if $\oo(e)=0$ (resp.\ $\oo(e)=1$).

By Bernoulli, bond \defi{percolation} on $G$ with parameter $p\in [0,1]$ we mean the random subgraph of $G$ obtained by keeping each edge with probability $p$ and deleting it with probability $1-p$, with these decisions being independent of each other.

The \defi{percolation threshold} $p_c(G)$ is defined by 
$$p_c(G):= \sup\{p \mid \Pr_p(|C_o|=\infty)=0\},$$
where the \defi{cluster} $C_o$ of $o\in V$ is the component of $o$ in the subgraph of \g spanned by the occupied edges. It is well-known that $p_c(G)$ does not depend on the choice of $o$.

To define \defi{site percolation} we repeat the same definitions, except that we now let $\Omega:= \{0,1\}^V$, and let $C_o$ be the component of $o$ in the subgraph of \g induced by the occupied vertices. The percolation threshold for site percolation is denoted $\pcs$.

In this paper the graphs $G$ we consider are all countably infinite, connected and every vertex has finite degree. Some of our results will need assumptions on $G$ like (quasi-)vertex-transitivity or planarity, but these will be explicitly stated as needed.

\subsection{Graph theoretic definitions} \label{graph defs}

Let $G = (V,E)$ be a graph. An \defi{induced} subgraph $H$ of $G$ is a subgraph that contains all edges $xy$ of $G$ with $x,y\in V(H)$. Note that $H$ is uniquely determined by its vertex set. The subgraph of $G$ \defi{spanned} by a vertex set $S\subseteq V(G)$ is the induced subgraph of $G$ with vertex set $S$. The vertex set of a graph \g will be denoted by $V(G)$, and its edge set by $E(G)$.

The \defi{edge space} of a graph \g is the direct sum $\ce(G):= \bigoplus_{e\in E(G)} \Z_2$, where $\Z_2=\{0,1\}$ is the field of two elements, which we consider as a vector space over $\Z_2$. 
The \defi{cycle space} $\cc(G)$ of \g is the subspace of $\ce(G)$ spanned by the \defi{circuits} of cycles, where a circuit is an element $C\in \ce(G)$ whose non-zero coordinates $\{e\in E(G) \mid C_e=1\}$ coincide with the edge-set of a cycle of \G. 

A \defi{planar graph} $G$ is a graph that can be embedded in the plane $\R^2$, i.e.\ it can be drawn in such a way that no edges cross each other. Such an embedding is called a \defi{planar embedding} of the graph. A \defi{plane graph} is a (planar) graph endowed with a fixed planar embedding.

A plane graph divides the plane into regions called \defi{faces}. Using the faces of a plane graph $G$ we define its \defi{dual graph} $G^*$ as follows. The vertices of $G^*$ are the faces of $G$, and we connect two vertices of $G^*$ with an edge whenever the corresponding faces of $G$ share an edge. Thus there is a bijection $e\mapsto e^*$ from $E(G)$ to $E(G^*)$.

Given a subgraph $H$ of a graph $G$ and a positive integer $k$, we define the \defi{k-neighbourhood} of $H$ to be the set of vertices at distance at most $k$ from $H$.

\subsection{Partitions}

A \defi{partition} of a positive integer $n$ is a multiset $\{m_1,m_2,\ldots,m_k\}$ of positive integers such that $m_1+m_2+\ldots+m_k=n$. Let $p(n)$ denote the number of partitions of $n$. An asymptotic expression for $p(n)$ was given by Hardy \& Ramanujan in their famous paper \cite{HarRam}. An elementary proof of this formula up to a multiplicative constant was given by Erd\H os \cite{ErdHR}. As customary we use $A\sim B$ to denote the relation $A/B\rightarrow 1$ as $n\rightarrow \infty$.

\begin{theorem}[Hardy-Ramanujan formula] \label{HRthm}
The number $p(n)$ of partitions of $n$ satisfies
$$p(n)\sim \dfrac{1}{4n\sqrt{3}}\exp\Big(\pi \sqrt{\dfrac{2n}{3}}\Big).$$
\end{theorem}

(We do not need the full strength of \Tr{HRthm} in this paper; any sub-exponential upper bound on $p(n)$ would suffice, and such bounds are much easier to obtain. See e.g.\ \cite[Lemma 3.4]{analyticity}.)

\subsection{Quasi-transitive planar lattices}\label{qtl}

In this subsection we will consider graphs that embed in $\mathbb{R}^2$, in a `nice' way.

\begin{definition}
A \defi{\qtl} is a locally finite, connected graph $G$ embedded in $\R^2$ \st\ for some linearly independent vectors $v_1,v_2\in \R^2$, translation by each $v_i$ preserves $G$, and the action defined by the translations has finitely many orbits of vertices. 
\end{definition}

Although not part of the definition, we will always assume that \qtl s are $2$-connected. This is only a minor assumption because the boundary of a face of $G$ contains a cycle that surrounds every other boundary vertex of the same face. By deleting every vertex that does not lie in the surrounding cycle of some face of $G$, we obtain a $2$-connected \qtl\ with the same $p_c$ as the initial graph.

\begin{definition}
Given a finite subgraph $H$ of an infinite graph $G$, the \defi{minimal edge cut} $\partial^E H$ of $H$ is defined to be the minimal set of edges lying in $E(G)\setminus E(H)$ with at least one endvertex in $H$, the removal of which disconnects $H$ from infinity. The \defi{minimal vertex cut} $\partial^V H$ of $H$ is the minimal set of vertices in $G\setminus H$ that are incident to $H$, the removal of which disconnects $H$ from infinity.
\end{definition}

It is not hard to see that \qtl s are quasi-isometric to $\R^2$, inheriting some of its geometric properties. More precisely any \qtl\ $G$
\begin{enumerate}[(1)]
\item has quadratic growth, i.e.\ there are constants $c_1=c_1(G),c_2=c_2(G)>0$ such that
$$c_1 n^2\leq |B(u,n)|\leq c_2 n^2$$
for every $u\in V(G)$ and every positive integer $n$, where $B(u,n)$ denotes the ball of radius $n$ around $u$ in either graph-theoretic distance or euclidean distance,
\item satisfies a $2$-dimensional isoperimetric inequality, i.e.\ there is a constant $c=c(G)>0$ such that for any finite subgraph $H\subset G$,
$$|\partial^V H|\geq c \sqrt{|H|}.$$
\end{enumerate}

It will be useful to define a more general type of isoperimetric inequality.
\begin{enumerate}[(1)]
\setcounter{enumi}{2}
\item \label{isop} Given a positive number $d$ (not necessarily an integer), we say that a graph $G$ satisfies a $d$-dimensional isoperimetric inequality if there is a constant $c>0$ such that for any finite subgraph $H\subset G$,
$$|\partial^V H|\geq c |H|^{\frac{d-1}{d}}.$$
\end{enumerate}

\noindent
Any \qtl\ $G$ is easily seen to satisfy the following properties as well:
\begin{enumerate}[(1)]
\setcounter{enumi}{3}
\item \label{quasi-geodesic} For some $o\in V(G)$, there is a $2$-way infinite path $X=(\ldots,x_{-1},x_0=o,x_1,\ldots)$ containing $o$ and a constant $l>0$, such that 
$d_X(x_i,x_j)\leq l d_G(x_i,x_j)$ for every $i,j\in\mathbb{Z}$, where $d_X$ and $d_G$ denote distance in $X$ and $G$, respectively. Moreover, it is not too hard to see that we can choose $X$ to be periodic, i.e.\ to satisfy $X+ t v_1= X$ for some $t\in \N$.
Any such path is called a \defi{quasi-geodesic}. 
\item \label{finitely presented} The cycle space of $G$ is generated by cycles of bounded length.
\item \label{$1$-ended} $G$ is \defi{$1$-ended}, i.e.\ for every finite subgraph $H$ of $G$, the graph $G\setminus H$ has a unique infinite component.
\end{enumerate}

\section{Interfaces} \label{pints site-pints}

In this section we recall the notions of (bond-)\pint s and site-\pint s introduced in \cite{analyticity}. In most cases, we will work with the following families of graphs:
\begin{enumerate}[(a)]
\item \qtl s,
\item the standard cubic lattice $\Z^d$, $d>1$,
\item $\mathbb{T}^d$, the graph obtained by adding to $\Z^d$, $d>1$ the `monotone' diagonal edges, i.e.\ the edges of the form $xy$ where $y_i-x_i=1$ for exactly two coordinates $i\leq d$, and $y_i=x_i$ for all other coordinates ($\mathbb{T}^2$ is isomorphic to the triangular lattice).
\end{enumerate} 
Let us denote with $\mathcal{S}$ the set of all those graphs.

For each $G\in \cs$ we will fix a basis $\cp=\cp(G)$ of the cycle space $\cc(G)$ (defined in \Sr{graph defs}). If $G$ is a \qtl\ , $\cp$ consists of the cycles bounding the faces of $G$. For $G=\Z^d$ we can use the squares bounding the faces of its cubes as our basis $\cp$, and for 
$G=\mathbb{T}^d$ we can use the triangles obtained from the squares once we add the `monotone' diagonal edges. Our definition of the \defi{\pint} of \g depends on the choice of $\cp(G)$, and so in \cite{analyticity} we used the notation `\cp-interface' to emphasize the dependence. Since in this paper we are fixing $\cp(G)$ for each $G\in \cs$, we will simplify our notation and just talk about \pint s. 

Let us start by defining \pint s for \qtl s.

\begin{definition} 
Let $G$ be a \qtl\ and $o$ a vertex of $G$. A subgraph $P$ of $G$ is called a \defi{(bond-)\pint} (of $o$) if there is a finite connected subgraph $H$ of $G$ containing $o$ such that $P$ consists of the vertices and edges incident with the unbounded face of $H$. The \defi{boundary $\partial P$} of $P$ is the set of edges of $G$ that are incident with $P$ and lie in the unbounded face of $H$. We say that an \pint\ \defi{occurs} in a bond percolation instance \oo\ if all edges in $P$ are occupied and all edges in $\partial P$ are vacant.
\end{definition}

As remarked in \cite{analyticity} \pint s are connected graphs, and satisfy the following property.
\begin{lemma} \label{scs disjoint} 
For any graph $G$, if two occurring \pint s of $G$ share a vertex then they coincide. 
\end{lemma}

With some thought this notion can be generalised to higher dimensions in such a way that a unique \pint\ is associated to any cluster. The reader may already have their own favourite definition of \pint\ for $G=\Z^d$ or $G=\mathbb{T}^d$, and as long as that definition satisfies \Tr{unique} below it will coincide with ours. For the remaining readers we offer the following abstract definition. For (site percolation on) $G=\mathbb{T}^d$ we offer a simpler alternative definition implicit in \Prr{connected}.

\setcounter{equation}{6}

\medskip
To define \pint s in full generality, we need to fix first some notation.
From now on, we will be assuming that 
\labtequ{assumptions}{\g is an $1$-ended and $2$-connected graph.} Every edge $e=vw\in E(G)$ has two \defi{directions} $\ar{vw},\ra{wv}$, which are the two directed sets comprising $v,w$. The head $head(\ar{vw})$ of $\ar{vw}$ is $w$. Given $F\subset E(G)$ and a subgraph $D$ of $G$, let $\ard{F}{D}:= \{\ar{vz} \mid vz\in F, z\in V(D) \}$ be the set of directions of the elements of $F$ towards $D$.

Let $\cp$ denote a basis of $\cc(G)$ (which we fixed at the beginning of this section). A \defi{\cp-path connecting} two directed edges $\ar{vw}, \ra{yx} \in \overleftrightarrow{E(G)}$ is a path $P$ of \g \st\ the extension $vw P yx$ is a subpath of an element of \cp. Here, the notation $vw P yx$ denotes the path with edge set $E(P) \cup \{vw, yx\}$, with the understanding that the endvertices of $P$ are $w,y$. Note that $P$ is not endowed with any notion of direction, but the directions of the edges $\ar{vw}, \ra{yx}$ it connects do matter. We allow $P$ to consist of a single vertex $w=y$.

We will say that $P$ \defi{connects} an undirected edge $e\in E(G)$ to $\ar{f}\in \overleftrightarrow{E(G)}$ (respectively, to a set $J\subset \overleftrightarrow{E(G)}$), if $P$ is a \cp-path connecting one of the two directions of $e$ to $\ar{f}$ (resp.\ to some element of $J$).

\begin{definition} \label{def con dir}
We say that a set ${J}\subset \overleftrightarrow{E(G)}$ is \defi{$F$-connected} for some $F\subset E(G)$, if \fe\ proper bipartition $ ({J_1},{J_2})$ of ${J}$, there is a \cp-path in $G\setminus F$ connecting an element of ${J_1}$ to an element of ${J_2}$.
\end{definition}

We are now ready to give the formal definition of the central notion of the paper. 

\begin{definition}\label{pint}
A \defi{(bond-)\pint} of \g is a pair $(P,\partial P)$ of sets of edges of \g with the following properties
\begin{enumerate}
\item \label{pint a} $\partial P$ separates $o$ from infinity; 
\item \label{pint x} There is a unique finite component $D$ of $G\setminus \partial P$ containing a vertex of each edge in $\partial P$;
\item \label{pint c} 
\ard{\partial P}{D} is $\partial P$-connected; and
\item \label{pint b} $P = \{e \in E(D) \mid \text{ \ti\ a \cp-path in $G\setminus \partial P$ connecting $e$ to $\ard{\partial P}{D}$ } \}$.
\end{enumerate}
\end{definition}

We say that an \pint\ $(P,\partial P)$ occurs in a bond percolation instance $\omega$ if the edges of $P$ are occupied, and the edges of $\partial P$ are vacant. 

(Bond-)\pint s are specifically designed to study bond percolation on $G$. There is a natural analogue for site percolation. For an \pint\ $(P,\partial P)$ of \g we let $V(P)$ denote the set of vertices incident with an edge in $P$, and we let $V(\partial P)$ denote the set of vertices incident with an edge in $\partial P$ but with no edge in $P$. We say that an \pint\ $(P,\partial P)$ is a \defi{site-\pint}, if no edge in $\partial P$ has both its end-vertices in $V(P)$. We say that a site-\pint\ $(P,\partial P)$ occurs in a site percolation instance $\omega$ if the vertices of $V(P)$ are occupied, and the vertices of $V(\partial P)$ are vacant. We will still use $P$ and $\partial P$ to refer to $V(P)$ and $V(\partial P)$.

We say that $(P,\partial P)$ meets a cluster $C$ of $\omega$, if either $P\cap E(C)\neq \emptyset$, or $P = E(C) =\emptyset$ and $\partial P=\partial C$, where $\partial C$ is the set edges in $E(G)\setminus E(C)$ with at least one endvertex in $C$ (in which case $C$ consists of $o$ only).

\medskip
The following result applies to both bond- and the site-interfaces.
\begin{theorem}[{\cite[Theorem~10.4.]{analyticity}}] \label{unique}
For every finite (site) percolation cluster $C$ of $G$ such that $C$ separates
$o$ from infinity, there is a unique (site-)\pint\ $(P,\partial P)$ that meets $C$ and occurs. Moreover, we have $P\subset E(C)$ and $\partial P\subset \partial C$.

Conversely, every occurring (site-)\pint\ meets a unique percolation cluster
$C$, and $\partial C$ separates $o$ from infinity (in particular, $C$ is finite).
\end{theorem}

This allows us to define the (site-)\pint\ of a cluster $C$ of a percolation instance $\oo$ as the unique occurring (site-)\pint\ that meets $C$.

\begin{remark}\label{t/2}
Let $G$ be a graph the cycle space of which admits a basis consisting of cycles of length bounded by some constant $t>0$. Then for every \pint\ $(P,\partial P)$ of $G$, and any pair of edges in $\partial P$, there is a path contained in the $t/2$-neighbourhood of $\partial P$ connecting the pair (see \cite[p. 47]{analyticity}).
\end{remark}

We define a \defi{\mpint} to be a finite collection of pairwise disjoint \pint s, and a \defi{site-\mpint} to be a finite collection of pairwise disjoint site-\pint s. An example is shown in \fig{figscv}.

\begin{figure}[htbp]
\vspace*{5mm}
\centering
\noindent

\begin{overpic}[width=.6\linewidth]{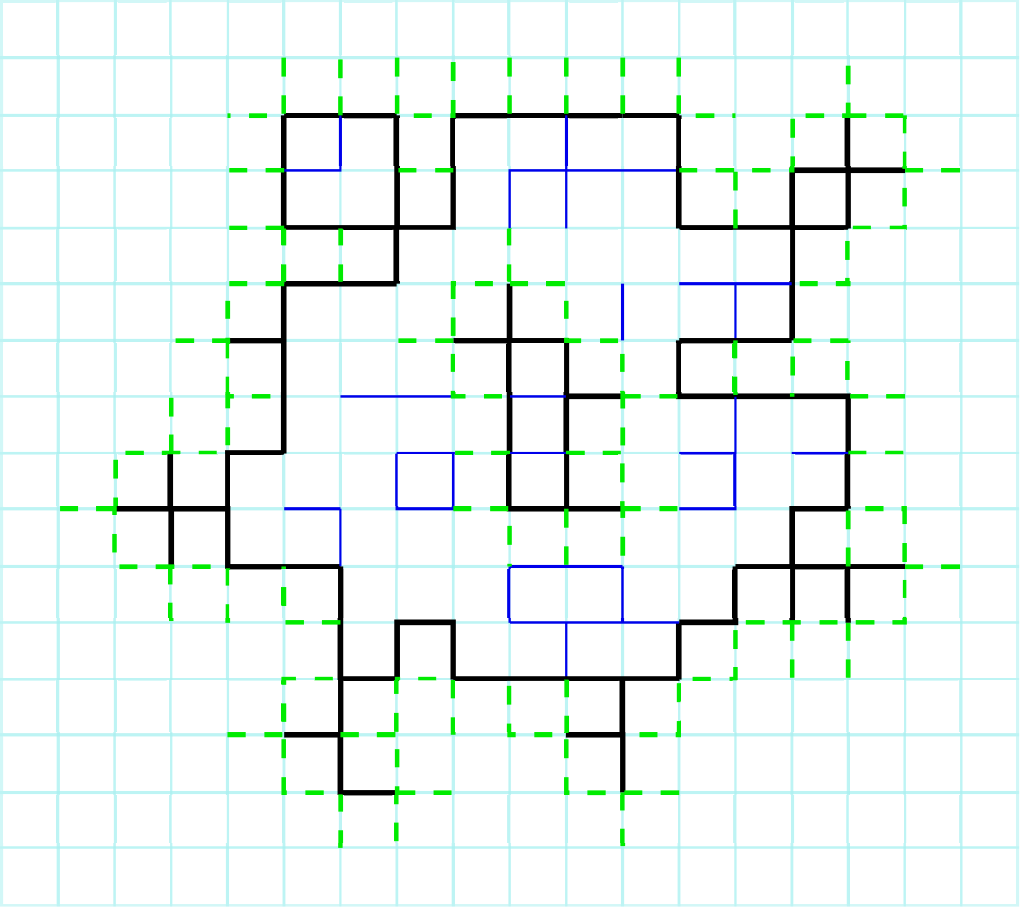}
\put(51,40){$o$}
\put(56,52){$P_1$}
\put(16,70){$P_2$}
\end{overpic}
\begin{minipage}[c]{0,9\textwidth}
\vspace*{5mm}
\caption{\small An example of a multi-interface $M$, comprising two nested interfaces $P_1,P_2$. We depict $M$ with bold lines, and $\partial M:= \partial P_1 \cup \partial P_2$ with dashed lines (green, if colour is shown). The edges not participating in $M$ are depicted in plain lines (blue, if colour is shown).} \label{figscv}
\end{minipage}
\end{figure}

In the case where \labtequ{triangulation}{$G\in \mathcal{S}$ is a graph the cycle space of which is generated by its triangles,} site-\pint s admit an equivalent definition that is more standard and easier to work with:

\begin{proposition}\label{connected}
Let $G\in \mathcal{S}$ be a graph satisfying \eqref{triangulation}, and let $D$ be a finite induced subgraph of $G$ containing $o$. Let $\overline{D}$ be the union of $D$ with the finite connected components of $G\setminus D$. Define $P$ to be the set of vertices of $\overline{D}$ which have a neighbour not in $\overline{D}$, and let $\partial P$ be the set of vertices of $G\setminus\overline{D}$ that have a neighbour in $\overline{D}$. Then $(P,\partial P)$ is the site-\pint\ of $D$. Moreover, any site-\pint\ can be obtained in this way.
\end{proposition} 
\begin{proof} 
Let us start from the last statement which is easier to prove. By definition, for every site-\pint\ $(P,\partial P)$, $G\setminus \partial P$ has a unique finite component $D$. Notice that $G\setminus D$ has no finite components, i.e.\ $\overline{D}=D$. Now $\partial P$ is the set of vertices of $G\setminus\overline{D}$ that have a neighbour in $\overline{D}$. Moreover, all vertices of $\overline{D}$ which have a neighbour not in $\overline{D}$ belong to $P$. On the other hand, each vertex $u$ of $P$ has a neighbour in $\partial P$ because there is a \cp-path connecting $u$ to $\partial P$ and $G$ satisfies \eqref{triangulation}.

Let $L$ be the set of edges with one endvertex in $P$ and another in $\partial P$, and let $Q$ be the set of edges $e$ in $E(D)$ such that there is a \cp-path in $G\setminus L$ connecting $e$ to \ard{L}{D}. If $e$ is an edge in $Q$, then both its endvertices are incident with an edge in $L$, hence both of them lie in $P$, because \cp\ contains only triangles, and $D$ is an induced subgraph. It suffices to show that $(Q,L)$ is the \pint\ of $D$, as this will immediately imply that $(P,\partial P)$ is the corresponding site-\pint\ of $D$. 

It is easy to verify that $L$ satisfies the first two items of Definition \ref{pint}, and that $Q$ satisfies the last item of the definition. It remains to prove that \ard{L}{D} is $L$-connected. Assuming not, we find a proper bipartition $(L_1,L_2)$ of $L$, such that no \cp-path connects $L_1$ with $L_2$. Consider $e\in L_1$ and $f\in L_2$. Then there are two paths connecting $e$ with $f$, where one of them lies in $D$ and the other one lies in the complement of $D$. The union of the two paths with $e$ and $f$ is a cycle, which we denote $K$. 

Since \cp\ is a basis for the cycle space $\cc(G)$, $K$ can be expressed as a sum $\sum C_i$ of cycles $C_i\in \cp$. Let $L_{C_i}:= \overleftrightarrow{L \cap E(C)}$ be the directions of edges of $L$ appearing in $C_i$. Note that no cycle $C_i$ contains a path in $G\setminus L$ connecting $L_1$ to $L_2$, because no such path exists by the choice of $(L_1,L_2)$. Consequently, $L_{C_i}$ has an even number of its elements in each of $L_1,L_2$, because each component of $C_i\setminus L$ (which is a subpath of $C_i$) is incident with either 0 or 2 such elements pointing towards the component, and they lie both in $L_1$ or both in $L_2$ or both in none of the two.

This leads into a contradiction by a parity argument: notice that our cycle $K$ contains an odd number of directions of edges in each of $L_1,L_2$, namely exactly one in each ---$e$ and $f$ respectively--- because $P$ avoids $L$ and $Q$ avoids $D$, hence $\ard{L}{D}$, by definition. But then our equality $K=\sum C_i$ is impossible by the above claim because sums in $\cc(G)$ preserve the parity of the number of (directed) edges in any set. This contradiction proves our statement.
\end{proof}

Most of the time we will write $P$ instead of $(P,\partial P)$ to simplify the notation.

\section{Growth rates}\label{Growth rates}
In this section we give the formal definition of $b_r$ in its bond and site version, obtain some basic facts about it, and establish the connection to percolation.

\medskip
Given a graph $G$, we let $I_{n,r,\epsilon}=I_{n,r,\epsilon}(G)$ denote the set of \pint s $P$ with $|P|=n$ and $(r-\epsilon)n \leq |\partial P| \leq (r+\epsilon)n$. Here $|\cdot|$ counts the number of edges.
Similarly, we let $MI_{n,r,\epsilon}=MI_{n,r,\epsilon} (G)$ denote the set of \mpint s $P$ with $|P|=n$ and $(r-\epsilon)n \leq |\partial P| \leq (r+\epsilon)n$.

To avoid introducing a cumbersome notation, we will still write $I_{n,r,\epsilon}$ and $MI_{n,r,\epsilon}$ for the site-\pint s and site-\mpint s, respectively, of size $n$ and boundary size between $(r-\epsilon)n$ and $(r+\epsilon)n$. Moreover, we will write $c_{n,r,\epsilon}^\circ$ and $c_{n,r,\epsilon}^\odot$ for the cardinality of $I_{n,r,\epsilon}$ and $MS_{n,r,\epsilon}$, respectively. 

\

{\bf The definitions, results and proofs that follow apply to both\\ (bond-)\pint s and site-\pint s unless otherwise stated. }

\begin{definition} \label{def br}
Define the (upper) exponential growth rate $b_r^\circ(G)$ of the (bond- or site-) \pint s of $G$ with surface-to-volume ratio $r$ by
$$b_r^\circ=b_r^\circ(G):= \lim_{\epsilon \to 0} \limsup_{n\to\infty} {c_{n,r,\epsilon}^\circ(G)}^{1/n}.$$
Similarly we define the (upper) exponential growth rate $b_r^\circ(G)$ of the (site-)\mpint s of \g with surface-to-volume ratio $r$ by
$$b_r^\odot=b_r^\odot(G):= \lim_{\epsilon \to 0} \limsup_{n\to\infty} {c_{n,r,\epsilon}^\odot(G)}^{1/n}.$$
\end{definition}

We remark that in Hammond's definition of the exponential growth rate of lattice animals with surface-to-volume ratio $r$, $\epsilon$ depends on $n$. The above definition simplifies the proofs of some of the following results.

We are going to study $b_r^\circ$ and $b_r^\odot$ as functions of $r$. As it turns out, these two functions coincide:

\begin{lemma} \label{pint mpint}
Let $G\in \mathcal{S}$. Then $b_r^\circ(G) = b_r^\odot(G)$.
\end{lemma}

We postpone the proof until the next section where the necessary definitions and tools are introduced.

{\bf From now on, except for the proof of \Lr{pint mpint}, we will drop the superscripts and we will simply write $b_r$ and $c_{n,r,\epsilon}$}. In our proofs we will work with \pint s and site-\pint s instead of \mpint s and site-\mpint s.

Similarly to $b_r$, we define the (upper) exponential growth rate of all \pint s of $G$:
$$b=b(G):= \lim_{\epsilon \to 0} \limsup_{n\to\infty} c_{n}(G)^{1/n}$$
where $c_{n}(G):= |\{\text{ \pint s $P$ with $|P|=n$} \}|$. In the following proposition we prove that $b(G)=\max_r b_r(G)$.

\begin{proposition} \label{equality}
Let $G$ be a bounded degree graph. Then there is some $r$ such that $b(G)= b_r(G)$.
\end{proposition}
\begin{proof} 
Notice that there are no (site-)\pint s $P$ with $|\partial P|/|P|>\Delta$, where $\Delta$ is the maximum degree of $G$. Recursively subdivide the interval $I_0:= [0,\Delta]$ into two subintervals of equal length. At each step $j$, one of the two subintervals $I_j$ of $I_{j-1}$ accounts for at least half of the (site-)\pint s $P$ of size $n$ with $|\partial P|/|P|\in I_{j-1}$ for infinitely many $n$. Hence there are at least $2^{-j} c_n$ (site-)\pint s of size $n$ with $|\partial P|/|P|\in I_j$ for infinitely many $n$. By compactness, $[0,\Delta]$ contains an accumulation point $r_0$ of the $I_j, j\in \N$. Notice that for every $\epsilon>0$ we have $\limsup_{n\to\infty} c_{n,r_0,\epsilon}^{1/n}=b$. Taking the limit as $\epsilon$ goes to $0$, we obtain $b= b_{r_0}$, as desired.
\end{proof}

Using \Tr{unique} we can easily obtain some bounds for $b_r$.

\begin{proposition} \label{br ineq}
Let $G$ be a graph satisfying \eqref{isop} or \eqref{quasi-geodesic}. Let $r>0, 0\leq p\leq 1$. Then we have $p(1-p)^r \leq 1/b_r(G)$.
\end{proposition}
\begin{proof} 
Let us first assume that $G$ satisfies \eqref{quasi-geodesic}. Let $N_n$ be the (random) number of occurring (site-)\pint s $P$ with $|P|=n$ in a percolation instance $\omega$. Consider a quasi-geodesic $X$ containing $o$, and let $X^+$, $X^-$ be its two infinite subpaths starting from $o$. Any occurring (site-)\pint\ $P$ has to contain a vertex $x^+$ in $X^+$, and a vertex $x^-$ in $X^-$ ($x^+$ and $x^-$ may possibly coincide). If $|P|=n$, then $d_G(x^+,x^-)\leq n$, because $P$ is a connected graph. Hence $d_X(x^+,x^-)\leq ln$, implying that $x^+$ is one of the first $ln+1$ vertices of $X^+$. Since occurring (site-)\pint s are disjoint by \Tr{unique}, \begin{align}\label{occurrence}
N_n\leq ln+1
\end{align} 
for every $n$ and any bond (site) percolation instance $\omega$. Therefore, $\mathbb{E}_p(N_n)\leq ln+1$ for every $p\in [0,1]$. We now have 
$ln+1\geq \mathbb{E}_p(N_n)\geq c_{n,r,\epsilon} (p(1-p)^{r+\epsilon})^n$. Taking the $n$th root, and then letting $n$ go to infinity, and $\epsilon$ go to $0$, we obtain $p(1-p)^r \leq 1/b_r$, as desired.

If $G$ does not contain a quasi-geodesic but satisfies the isoperimetric inequality \eqref{isop}, then the assertion can be proved as follows. Since $G$ is locally finite, it contains an $1$-way infinite path $X$ starting from $o$ that does not revisit the same vertex twice. Any occurring (site-)\pint\ $P$ has to contain one of the first $(|P|/c)^{\frac{d}{d-1}}$ vertices of $X$ by \eqref{isop}, hence $N_n\leq (n/c)^{\frac{d}{d-1}}$. Arguing as above, we obtain $p(1-p)^r \leq 1/b_r$.
\end{proof}

Next, we observe that for any fixed $r$, equality in \Prr{br ineq} can occur for at most one value of $p$, which value we can compute:
\begin{proposition} \label{br eq}
Let $G$ be a graph satisfying \eqref{isop} or \eqref{quasi-geodesic}. If $p(1-p)^r = 1/b_r(G)$ for some $r,p$, then $p=\frac1{1+r}$ (and so $r=\frac{1-p}{p}$ and $1/b_r(G)= p(1-p)^\frac{1-p}{p}$).
\end{proposition}
\begin{proof} 
Fix $r$ and let $M:= \max_{p\in [0,1]} p(1-p)^r$. If $p_0(1-p_0)^r = 1/b_r$ is satisfied for some $p_0\in [0,1]$, then $p_0$ must attain $M$ by \Prr{br ineq}, that is, we have $M= p_0(1-p_0)^r$. Since the function $f(p)= p(1-p)^r$ vanishes at the endpoints $p=0,1$, we deduce that $f'(p_0)=0$. By elementary calculus, $f'(p)= (1-p)^r - rp(1-p)^{r-1}$, from which we obtain $r= \frac{1-p_0}{p_0}$ and $p_0=\frac1{1+r}$. 
\end{proof}

\medskip
Combining \Prr{br ineq} and \Prr{br eq} we obtain 
\labtequ{r ineq}{$b_r \leq f(r):= \frac{(1+r)^{1+r}}{r^r}.$}

\comment{
\begin{figure}[htbp] 
\centering
\noindent
\begin{overpic}[width=.6\linewidth]{br_vs_f(r).pdf}
\put(94,18){$b_r$}
\put(98,-1){$r$}
\put(54,-1){$r(p_c)$}
\put(21,-1){$r(1-p_c)$}
\put(94,48){$f(r)$}
\end{overpic}
\caption{\small The graph of $b_r$ (depicted in black colour if colour is shown) lies below the graph of $f$ (depicted in blue colour if colour is shown).}
\label{b_r vs f(r)}
\end{figure}
}

Motivated by \Prr{br eq}, we define the functions 
$$p(r):= \frac1{1+r} \text{ and } r(p):= \frac{1-p}{p}.$$ 
These functions are 1-1, strictly monotone decreasing, and the inverse of each other.

\

Recall that $N_n$ denotes the number of occurring \mpint s $P$ with $|P|=n$. The next result says that equality is achieved in \Prr{br eq} (for some $r$) exactly for those $p$ for which exponential decay in $n$ of $\Ex_p(N_n)$ fails.

\comment{
\begin{proposition} \label{no exp dec}
Let $G$ be a graph satisfying \eqref{isop} or \eqref{quasi-geodesic} and $p\in (0,1)$. Then $\Ex_p(N_n)$ fails to decay exponentially in $n$ if and only if 
$b_{r(p)}(G) = 1/p(1-p)^{r(p)}$ (that is, if and only if equality is achieved in \Prr{br eq}). 
\end{proposition}
\begin{proof} 
The backward implication is straightforward by the definitions.
\medskip

For the forward implication, suppose to the contrary that $$b_{r(p)} < 1/p(1-p)^{r(p)}.$$ The definition of $b_r$ implies that there are $\epsilon,\delta>0$ such that $$c_{n,r(p),\epsilon} p^n (1-p)^{n(r(p)-\epsilon)}\leq (1-\delta)^n$$ for all but finitely $n$. Hence, if we denote by $N_{n,r(p),\epsilon}$ the (random) number of occurring (site-)\mpint s $P$ with $|P|=n$ and $(r(p)-\epsilon)n \leq |\partial P| \leq (r(p)+\epsilon)n$, then for every large enough $n$, $$\Ex_p(N_{n,r(p),\epsilon})\leq c_{n,r(p),\epsilon} p^n (1-p)^{n(r(p)-\epsilon)}\leq (1-\delta)^n,$$ which implies the exponential decay in $n$ of $\Ex_p(N_{n,r(p),\epsilon})$.

On the other hand, we claim that $\Ex_p(N_n-N_{n,r(p),\epsilon})$ always decays exponentially in $n$. Indeed, consider the function $g(q,r)=q(1-q)^r$. Notice that for every fixed $r$ the function $g_r(q):=g(q,r)$ is maximised at $\frac{1}{1+r}$ and is strictly monotone on the intervals $[0,\frac{1}{1+r}]$ and $[\frac{1}{1+r},1]$. Recall that $p=\frac{1}{1+r(p)}$, and define 
$$s=s(p,\epsilon):=\frac{1}{1+r(p)+\epsilon}$$ and $$S=S(p,\epsilon):=\frac{1}{1+r(p)-\epsilon}.$$ It follows that there is a constant $0<c=c(p,\epsilon)<1$ such that $g(p,r(p)+\epsilon)\leq c g(s,r(p)+\epsilon)$ and $g(p,r(p)-\epsilon)\leq c g(S,r(p)-\epsilon)$, because $s<p<S$. Moreover, we have $$\Big(\frac{1-p}{1-s}\Big)^r \leq\Big(\frac{1-p}{1-s}\Big)^{r(p)+\epsilon}$$ whenever $r\geq r(p)+\epsilon$, and $$\Big(\frac{1-p}{1-S}\Big)^r \leq\Big(\frac{1-p}{1-S}\Big)^{r(p)-\epsilon}$$ whenever $r\leq r(p)-\epsilon$. This implies that $g(p,r)\leq c g(s,r)$ for every $r\geq r(p)+\epsilon$, and $g(p,r)\leq c g(S,r)$ for every $r\leq r(p)-\epsilon$. Summing over all possible \\ (site-)\mpint s $P$ with $|P|=n$ and $|\partial P|> (r+\epsilon)n$ or $|\partial P|< (r-\epsilon)n$ gives $$\Ex_p(N_n-N_{n,r(p),\epsilon})\leq c^n(\Ex_{s}(N_n)+\Ex_{S}(N_n)).$$ Since both $\Ex_{s}(N_n),\Ex_{S}(N_n)\leq ln+1$, we conclude that $E_p(N_n-N_{n,r(p),\epsilon})$, and hence $E_p(N_n)$, decays exponentially in $n$, which contradicts our assumption. Therefore, $b_{r(p)} = 1/p(1-p)^{r(p)}$.
\end{proof}
}

\begin{proposition} \label{no exp dec}
Let $G$ be a graph satisfying \eqref{isop} or \eqref{quasi-geodesic} and $p\in (0,1)$. Then $\Ex_p(N_n)$ fails to decay exponentially in $n$ if and only if 
$b_{r(p)}(G) = 1/p(1-p)^{r(p)}$ (that is, if and only if equality is achieved in \Prr{br eq}). 
\end{proposition}

For the proof of this we will make use of a large-deviation bound for $\Ex_p(N_n)$. In the percolation setup of \Prr{no exp dec}, we denote by $N_{n,r(p),\epsilon}$ the (random) number of occurring (site-)\mpint s $P$ with $|P|=n$ and\\ $(r(p)-\epsilon)n \leq |\partial P| \leq (r(p)+\epsilon)n$. The following lemma says that for any $p\in (0,1)$, most occurying \mpint s have a volume-to-surface ratio close to the value $r(p)$.

\begin{lemma}\label{large deviation}
Let $G$ be a graph satisfying \eqref{isop} or \eqref{quasi-geodesic} and $p\in (0,1)$. For every $\epsilon>0$, there exists $\delta>0$ such that $\Ex_p(N_n-N_{n,r(p),\epsilon})\leq e^{-\delta n}$ for every $n\in \N$.
\end{lemma}
\begin{proof}
Consider the function $g(q,r)=q(1-q)^r$. Notice that for every fixed $r$ the function $g_r(q):=g(q,r)$ is maximised at $\frac{1}{1+r}$ and is strictly monotone on the intervals $[0,\frac{1}{1+r}]$ and $[\frac{1}{1+r},1]$. Recall that $p=\frac{1}{1+r(p)}$, and define 
$$s=s(p,\epsilon):=\frac{1}{1+r(p)+\epsilon}$$ and $$S=S(p,\epsilon):=\frac{1}{1+r(p)-\epsilon}.$$ It follows that there is a constant $0<c=c(p,\epsilon)<1$ such that $g(p,r(p)+\epsilon)\leq c g(s,r(p)+\epsilon)$ and $g(p,r(p)-\epsilon)\leq c g(S,r(p)-\epsilon)$, because $s<p<S$. Moreover, we have $$\Big(\frac{1-p}{1-s}\Big)^r \leq\Big(\frac{1-p}{1-s}\Big)^{r(p)+\epsilon}$$ whenever $r\geq r(p)+\epsilon$, and $$\Big(\frac{1-p}{1-S}\Big)^r \leq\Big(\frac{1-p}{1-S}\Big)^{r(p)-\epsilon}$$ whenever $r\leq r(p)-\epsilon$. This implies that $g(p,r)\leq c g(s,r)$ for every $r\geq r(p)+\epsilon$, and $g(p,r)\leq c g(S,r)$ for every $r\leq r(p)-\epsilon$. Summing over all possible \\ (site-)\mpint s $P$ with $|P|=n$ and $|\partial P|> (r+\epsilon)n$ or $|\partial P|< (r-\epsilon)n$ gives $$\Ex_p(N_n-N_{n,r(p),\epsilon})\leq c^n(\Ex_{s}(N_n)+\Ex_{S}(N_n)).$$ Since both $\Ex_{s}(N_n),\Ex_{S}(N_n)\leq ln+1$, we conclude that $E_p(N_n-N_{n,r(p),\epsilon})$ decays exponentially in $n$.
\end{proof}

\begin{proof}[Proof of \Prr{no exp dec} ]
The backward implication is straightforward by the definitions.
\medskip

For the forward implication, suppose to the contrary that $$b_{r(p)} < 1/p(1-p)^{r(p)}.$$ The definition of $b_r$ implies that there are $\epsilon,\delta>0$ such that $$c_{n,r(p),\epsilon} p^n (1-p)^{n(r(p)-\epsilon)}\leq (1-\delta)^n$$ for all but finitely $n$. Hence for every large enough $n$, $$\Ex_p(N_{n,r(p),\epsilon})\leq c_{n,r(p),\epsilon} p^n (1-p)^{n(r(p)-\epsilon)}\leq (1-\delta)^n,$$ which implies the exponential decay in $n$ of $\Ex_p(N_{n,r(p),\epsilon})$. On the other hand, $\Ex_p(N_n-N_{n,r(p),\epsilon})$ decays exponentially in $n$ by \Lr{large deviation}. Therefore, $\Ex_p(N_n)$ decays exponentially in $n$ too.
\end{proof}

Let $S_o$ denote the (site-)\pint\ of the cluster $C_o$ of $o$ if $C_o$ is finite, and $S_o=\emptyset$ otherwise. We can now easily deduce that the statement of \Prr{no exp dec} holds for $P_p( |S_o|=n )$ in place of $\Ex_p(N_n)$, as stated in \Tr{S_o}, which we repeat here for convenience:

\begin{theorem}\label{S_o II}
Let $G\in \mathcal{S}$. Then for every $p\in (0,1)$, the cluster size distribution $\Pr_p( |S_o|=n )$ fails to decay exponentially in $n$ if and only if\\ $b_{r(p)}(G)= 1/p(1-p)^{r(p)} = f(r(p))$.\footnote{That is, if and only if equality is achieved in \Prr{br eq}.} 
\end{theorem}
\begin{proof}
Let $X$ be a quasi-geodesic containing $o$ such that $X+tv_1=X$ for some $t\in \N$, and let $X^+$ be one of the two infinite subpaths of $X$ starting at $o$. If $G$ is $\Z^d$ or $\mathbb{T}^d$, we just let $X$ be a geodesic. Notice that any (site-)\pint\ $P$ meets $X^+$ at some vertex $x^+$. Using a multiple $ktv_1$ of $tv_1$ for some integer $k$, we can translate $x^+$ to one of the first $M$ vertices of $X^+$, for some fixed $M>0$. It is not hard to see that $P+ktv_1$ is a (site-)\pint\ of $o$, i.e.\ it still separates $o$ from infinity. On the event $A=A(P):=\{\text{$P+ktv_1$ occurs}\} \cap\{\text{the subpath of $X^+$ between $o$ and $x^+ +ktv_1$ is open}\}$, we have $S_o=P+ktv_1$. Moreover, $$p^M \Pr_p(P \text{ occurs})\leq \Pr(A).$$ Summing over all (site-)\pint s of size $n$ with the property that the first vertex of $X^+$ they contain is $x^+$, we obtain
$$ p^M \sum \Pr_p(P \text{ occurs})\leq \Pr_p( |S_o|=n ),$$ where
the sum ranges over all such (site-)\pint s. Since there are at most $ln+1$ choices for the first vertex of $X^+$, summing over all possible $x^+$ we obtain
$$p^M \mathbb{E}_p(N_n)\leq (ln+1)\Pr_p( |S_o|=n ).$$ 
On the other hand, clearly $$\Pr_p( |S_o|=n )\leq \mathbb{E}_p(N_n).$$
Therefore, $\Pr_p(|S_o|=n)$ decays exponentially if and only if $\Ex_p(X_n)$ does. The desired assertion follows now from \Prr{no exp dec}.
\end{proof}

\section{Duality}\label{Duality}

The main aim of this section is the proof of \Tr{dual br} (\Tr{dual site}), and an analogous statement for planar bond percolation (\Tr{dual bond}).
In this section we study the properties of both \pint s and site-\pint s of graphs in $\mathcal{S}$. 

\medskip
If $G\in \mathcal{S}$ satisfies \eqref{triangulation}, we say that $(P,\partial P)$ is an \defi{\ipint} of $G$ if $(\partial P,P)$ is a site-\pint\ of $G$. We define $b_r^*$ similarly to $b_r$, except that we now count \ipint s instead of site-\pint s. Note that both $P$ and $\partial P$ span connected graphs in this case. Since this operation inverts the surface-to-volume ratio, we have
\labtequ{b star}{$b_r^*= b_{1/r}^r$.}
If $G$ is a \qtl, we say that $(P,\partial P)$ is an \ipint\ of $G$ if $(\partial P^*,P^*)$ is an \pint\ of $G^*$. Again define $b_r^*(G)$ similarly to $b_r(G)$, except that we now count \ipint s in the dual lattice $G^*$. Then \eqref{b star} still holds in this case. 

The main results of this section are:

\begin{theorem} \label{dual site}
Consider a graph $G\in \mathcal{S}$ satisfying \eqref{triangulation}. Then for the site-\pint s in $G$ we have $b_r(G)= (b_{1/r}(G))^r$ for every $r>0$.
\end{theorem}

\begin{theorem} \label{dual bond}
Consider a \qtl\ $G$. Then for the \pint s in $G$ and $G^*$ we have $b_r(G)= (b_{1/r}(G^*))^r$ for every $r>0$.
\end{theorem}

To prove Theorems \ref{dual site} and \ref{dual bond} we need the following concepts. Given a graph $G\in \mathcal{S}$, let $v_1,v_2,\ldots, v_d\in \R^d$ be some linearly independent vectors that preserve $G$, and let $\mathcal{B}$ be the box determined by $v_1,v_2,\ldots,v_d$. For $\Z^d$ and $\mathbb{T}^d$ we can choose $v_1,v_2,\ldots v_d$ to be the standard basis of $\R^d$. Given a (site-)\pint\ $P$ of $G$, let $\mathcal{T}$ be the set of translations $\mathcal{B}+k_1v_1+k_2v_2+\ldots k_dv_d$, $k_1,k_2,\ldots,k_d \in \mathbb{Z}$ of $\mathcal{B}$ that intersect $P\cup \partial P$. The \defi{box} $B(P)$ of $P$ is the smallest box with sides parallel to $v_1,v_2,\ldots,v_d$ containing $\mathcal{T}$. We write $|B(P)|$ for the surface area of $B(P)$ and we call it the box size of $P$.

For every $n\in \N$, $r \in \R_+$, $\epsilon \in \R_+$ and $\delta \in \R_+$ let $c_{n,r,\epsilon,\delta}(G)$ denote the number of interfaces $P$ with $|P|=n$, $(r-\epsilon)n \leq |\partial P| \leq (r+\epsilon)n$ and $|B(P)|\leq \delta n$. We define $\widetilde b_r$ to be the exponential growth rate as $\delta \to 0$ and $\epsilon \to 0$:
$$\widetilde{b_r}=\widetilde{b_r}(G):= \lim_{\epsilon \to 0} \lim_{\delta \to 0}\limsup_{n\to\infty} {c_{n,r,\epsilon,\delta}(G)}^{1/n}.$$

Our aim now is to prove that $\widetilde b_r = b_r$. In other words, (site-)\pint s with a `fractal' shape have the same exponential growth rate as all (site-)\pint s. We will first consider the cases of $\Z^d$ and $\mathbb{T}^d$. 

\begin{figure}[htbp] 
\centering
\noindent
\begin{overpic}[height=.5\linewidth, width=.5\linewidth]{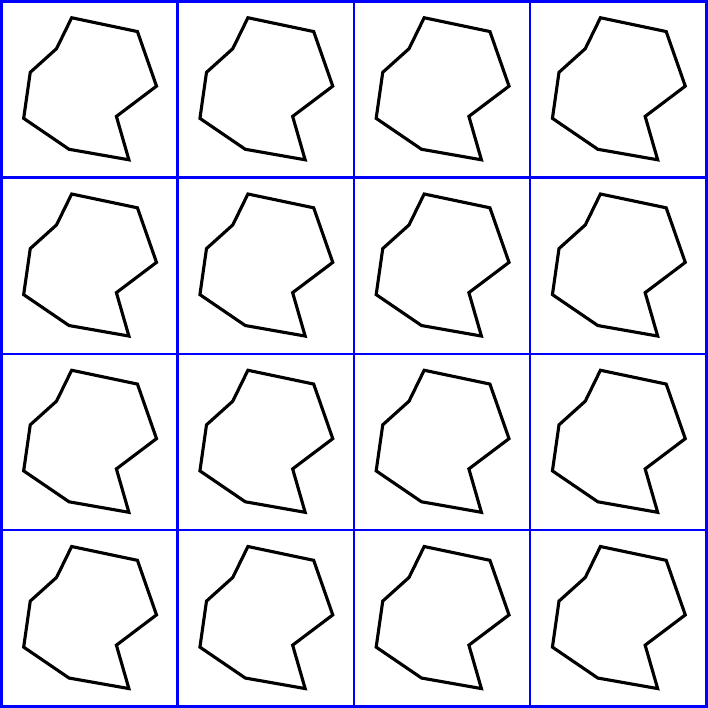}
\end{overpic}
\caption{\small The grid $B_n$ and the \pint s when $d=2$ in the proof of \Prr{o box}.}
\label{box}
\end{figure}

\begin{proposition} \label{o box}
Let $G$ be either $\Z^d$ or $\mathbb{T}^d$. Then $\widetilde b_r(G) = b_r(G)$.
\end{proposition}
\begin{proof}
We will first prove the assertion for \pint s. Let us first assume that $b_r>1$. Let $n\in\N$, $\epsilon>0$, $r>0$, and let $P\in I_{n,r,\epsilon}$. Consider the box $B(P)$ and notice that it contains $P\cup \partial P$ in its interior (no vertex of $P\cup \partial P$ lies in its topological boundary). Order the vertices on the boundary of the box arbitrarily. Define the \defi{shape} of an \pint\ $P$ to be the $3d$-tuple of numbers consisting of the dimensions of the box $B(P)$, and the $2d$ numbers determining the first vertex of each of the $2d$ faces of $B(P)$ in our ordering that is incident to $\partial P$. We will call these vertices \defi{extremal}. Notice that the extremal vertices are incident to a vertex in $\partial^V P$.

Notice that each side of $B(P)$ has at most $n+4$ vertices because the diameter of $P$ is at most $n$. Once we know the box of $P$, there are at most $(n+4)^{2d(d-1)}$ possibilities for the extremal vertices, hence there are at most $p(n):=(n+4)^{2d^2}$ possible shapes for \pint s of size $n$. On the other hand, there are exponentially many \pint s in $I_{n,r,\epsilon}$, so we can choose $n$ large enough to ensure that there is a non-empty set $K\subseteq I_{n,r,\epsilon}$ of cardinality at least $$N:=c_{n,r,\epsilon}/p(n)$$ consisting of \pint s $P$ with $|P|=n$ and $(r-\epsilon)n \leq |\partial P| \leq (r+\epsilon)n$ that have the same shape. 

We now piece elements of $K$ together to construct a large number of \pint s of arbitrarily high size that will contribute to $b_r$. We will construct a set $K_n$ of cardinality about $N^{n^{d(d-1)}}$ of \pint s of size about $n^{d(d-1)+1}$, of surface-to-volume ratio about $r$, and of small box size.

Recall that all \pint s in $K$ have the same shape, in particular, the same box $\mathcal{B}$. Let $B_n$ be a $d$-dimensional grid of $n^{d(d-1)}$ adjacent copies $B_{\textbf{i}}$, $\textbf{i}=(i_1,\ldots,i_d)$ of $B$ (each side contains $n^{d-1}$ copies of $B$). See \fig{box}. In each copy of $B$ in $B_n$, we place an arbitrary element of $K$. We denote the copy of $B$ placed in $B_{\textbf{i}}$ with $K_{\textbf{i}}$. Write $S_k$, $1\leq k\leq n^{d-1}$ for the slab containing the boxes $B_{\textbf{i}}$ with $i_1=k$. Our aim is to connect the \pint s inside the boxes using mostly short paths. First, consider $S_2$ and notice that every box in $S_2$ shares a common face with a box in $S_1$. We can move $S_2$ using the vectors $v_2,\ldots,v_d$ in order to achieve that the `rightmost' extremal vertices of $S_1$ coincide with the corresponding `leftmost' extremal vertices of $S_2$ lying in a common face with them. This is possible because all \pint s in $K$ have the same shape. Moving each slab $S_k$ in turn, we can make the `rightmost' and `leftmost' extremal vertices of consecutive slabs coincide. We now connect all these extremal vertices with their corresponding \pint s by attaching paths of length two parallel to $v_1$. Finally, we connect the \pint s in the first slab as follows. If two boxes in the first slab share a common face, then we connect the two extremal vertices lying in the common face with a path of minimum length inside that face (hence of length $O(n)$). Also, we attach a path of length two connecting all those extremal vertices to the \pint\ of their box (\fig{staircase}). 

\begin{figure}[htbp] 
\centering
\noindent
\begin{overpic}[width=.5\linewidth]{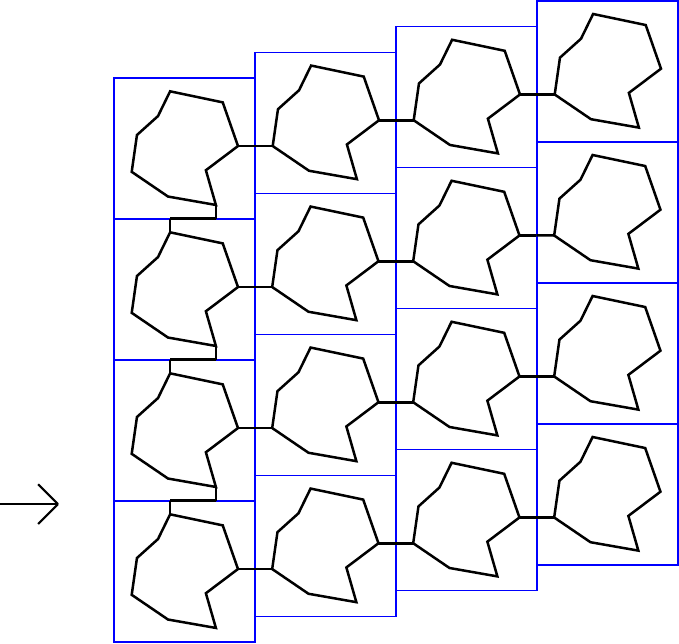}
\put(1,13){$u_1$}
\put(25,-7){$S_1$}
\put(46,-7){$S_2$}
\put(67,-7){$S_3$}
\put(88,-7){$S_4$}
\end{overpic}
\vspace*{5mm}
\caption{\small The \pint\ $Q$ in the proof of \Prr{o box}.}
\label{staircase}
\end{figure}

This construction defines a new graph $Q$. We claim that $Q$ is an \pint .
Indeed, if $d=2$ this follows easily from the topological definition of \pint s. For $d>2$, since $Q$ is a connected graph, there is an \pint\ associated to it. We will verify that $Q$ coincides with its \pint . 
Let $B$ denote the minimal edge cut of $Q$. Consider an \pint\ $K_{\textbf{i}}$ and the smallest box $B'_{\textbf{i}}$ containing $K_{\textbf{i}}$ (not necessarily its boundary). Any boundary edge in $\partial^E K_{\textbf{i}}$ that has not been attached to $Q$ has either one endvertex in the boundary of $B'_{\textbf{i}}$ and one in the complement of $B'_{\textbf{i}}$, or it can be connected in $B_{\textbf{i}}$ with a path lying in $G\setminus K_{\textbf{i}}$ with such an edge. From there it can be connected to infinity without intersecting $Q$. Hence all these boundary edges lie in $B$. If $e$ is an edge in $K_{\textbf{i}}\cup \partial K_{\textbf{i}}$ lying in the same basic cycle with one of the boundary edges attached to $Q$, then $e$ either lies in $\partial^E K_{\textbf{i}}$ or is incident to some edge in $\partial^E K_{\textbf{i}}$ that has not been attached to $Q$. This easily shows that every edge in $\partial K_{\textbf{i}} \setminus Q$ lies in the same $\partial Q$ component of \ard{\partial Q}{Q} with some edge in $B$. Furthermore, the boundary of the \pint\ of $Q$ coincides with the $\partial Q$ component of \ard{B}{Q}. Hence $\partial K_{\textbf{i}} \setminus Q$ lies in the boundary of the \pint\ of $Q$. It follows that all edges of $K_{\textbf{i}}$ lie in the \pint\ of $Q$. Finally, all attached edges are incident to an edge in $B$, which implies that they lie in the \pint\ of $Q$ as well. Therefore, $Q$ is contained in its \pint\, which proves that $Q$ coincides with it.

It can be easily seen that $Q$ has size roughly $n^{d(d-1)+1}$ and boundary size $$(r-\epsilon')|Q|\leq |\partial Q |\leq (r+\epsilon')|Q|$$ for some $\epsilon'=\epsilon'(n)$ not necessarily equal to $\epsilon$. Clearly we can choose $\epsilon'=\epsilon+o(1)$, since the number of attached edges is $O(n^d)$. The number of such $Q$ we construct is $|K|^{n^{d(d-1)}}\geq N^{n^{d(d-1)}}$, because by deleting all attached paths we recover all $K_{\textbf{i}}$, and we have $|K|$ choices for each $K_{\textbf{i}}$.

Note that each slab $S_k$ has been moved at distance at most $(k-1)n=O(n^d)$ from its original position. Hence, $|B(Q)|=O(n^{d(d-1)})=o(|Q|)$. The result follows by letting $n \to \infty$ and then $\epsilon \to 0$. 

It remains to consider the case where $b_r=0$ or $b_r=1$. If $b_r=0$, there is nothing to prove. If $b_r=1$, then we can argue as in the case $b_r>1$, except that now we place the same \pint\ at each box of the grid.

Let us now consider the case of site-\pint s. Let us focus on the case $b_r>1$. Let $K$ be a collection of at least $N$ site-\pint s of $I_{n,r,\epsilon}$, all of which have the same shape. Arguing as above, we place the elements of $K$ in a $d$-dimensional grid and we connect them in the same fashion to obtain a graph $Q$. For $\Z^d$ nothing changes, since $Q$ is an induced graph. However, this is not necessarily true for $\mathbb{T}^d$ because some endvertices of the attached paths are possibly incident to multiple vertices of the same site-\pint\ (even the paths of the first slab are not induced). This could potentially lead to an issue in the case that some boundary vertices cannot connect to infinity without intersecting the vertices of $Q$. 

But this is impossible in our case. Indeed, define $B''_{\textbf{i}}$ to be the smallest box containing $K_{\textbf{i}} \cup \partial K_{\textbf{i}}$. Notice that every face of $B''_{\textbf{i}}$ contains at most one vertex of $Q$. Hence no boundary vertex of $K_{\textbf{i}}$ lying in some face of $B'_{\textbf{i}}$ can be separated from infinity by the vertices of $Q$. Since any boundary vertex of $K_{\textbf{i}}$ can be connected in $G\setminus K_{\textbf{i}}$ to the boundary of $B'_{\textbf{i}}$, the claim follows. Thus the graph spanned by $Q$ is a site-\pint, which proves that $\widetilde b_r=b_r$ for site-\pint s as well. 
\end{proof}

The above arguments can be carried out for \pint s of any \qtl\ with only minor modifications that we will describe in \Lr{o box site}. However, certain difficulties arise when studying site-\pint s on an arbitrary \qtl . Indeed, when we connect two site-\pint s $P_1,P_2$ with a path, it is possible that some of the vertices of $\partial P_1$ or $\partial P_2$ are now `separated' from the remaining boundary vertices, see \fig{bad}. In fact, it is possible that most boundary vertices have this property. To remedy this, instead of choosing arbitrarily the path that connects $P_1$ and $P_2$, we will choose it appropriately so that only a few of them, if any, are `separated' from the remaining boundary vertices. 

\begin{figure}[htbp] 
\centering
\noindent
\begin{overpic}[width=.5\linewidth]{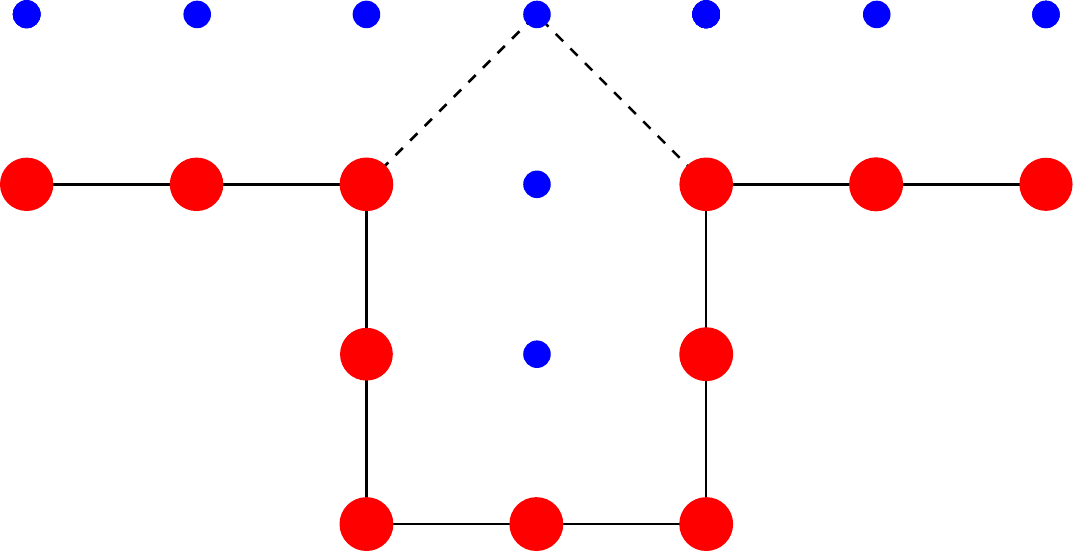}
\end{overpic}
\caption{\small If the vertex incident to the two dashed lines is attached to the site-\pint , the vertices of which are depicted with big disks, then the new graph is not a site-\pint\ any more.}
\label{bad}
\end{figure}

\begin{lemma}\label{connection}
Let $G$ be a graph satisfying \eqref{isop} for some $c>0,d\geq 2$. Assume that $\cc(G)$ admits a basis consisting of cycles whose length is bounded by some $t>0$. Let $P$ be a site-\pint\ of $G$. Then there are $|\partial^V P|-O(|P|^{1/4})=\Omega(|P|^{\frac{d-1}{d}})$ vertices $u\in \partial^V P$ such that the site-\pint\ of $P\cup \{u\}$ has size $|P|-O(|P|^{3/4})$ and boundary size $|\partial P|-O(|\partial P|^{3/4})$. 
\end{lemma}
\begin{proof}
For every $v\in \partial^V P$, let $P_v$ be the site-\pint\ of the connected graph $P\cup \{v\}$ and let $Q_v:=\partial P \setminus (\partial P_v \cup \{v\})$. Write $L$ for the edges with one endpoint in $P$ and the other in $\partial P$, $E_v$ for the edges of the form $vw$, $w\in P$, and $L_v$ for the edges with one endvertex in $P$ and the other in $Q_v$. First, we claim that the $Q_v$ are pairwise disjoint. Indeed, assuming that this is not true, we find a pair of distinct $u,v$ such that $Q_u\cap Q_v \neq \emptyset$. Since the vertices of $Q_z$, $z\in\{u,v\}$ do not belong to $\partial P_z$, \ard{E_z}{P} separates \ard{L_z}{P} from the remaining edges of \ard{L}{P}. Hence no vertex of $Q_z$ lies in $\partial^V P$, as any path starting from a vertex of $Q_z$ and going to infinity without intersecting $P$ must intersect $z$. This implies that if $X,Y$ are two overlapping components of \ard{L_u}{P}, \ard{L_v}{P}, respectively, then $X\cup Y$ is $L\setminus (E_u\cup E_v)$-connected, and thus $X,Y$ coincide. Moreover, $X$ is connected to \ard{E_u}{P} with a \cp-path in $G\setminus L$, and $Y$ is connected to \ard{E_v}{P} with a \cp-path in $G\setminus L$. Therefore, $u$ coincides with $v$, which is absurd. Hence, our claim is proved.

\begin{figure}[htbp] 
\centering
\noindent
\begin{overpic}[width=.6\linewidth]{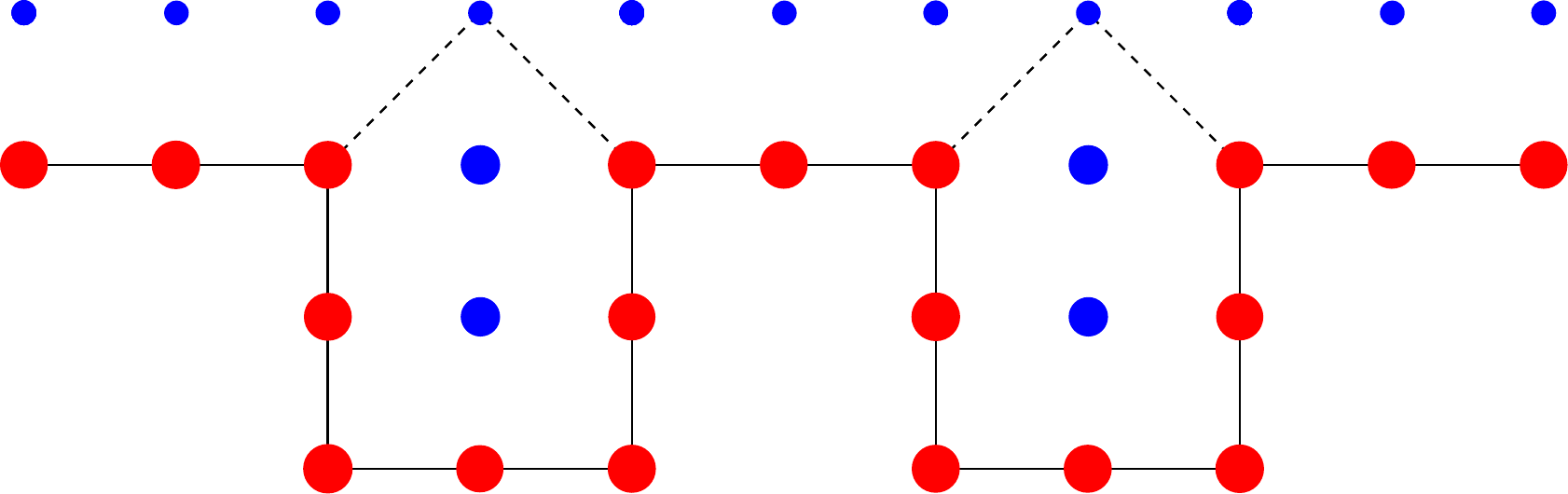}
\put(25,30){$u$}
\put(65,30){$v$}
\put(23,15){$Q_u$}
\put(63,15){$Q_v$}
\end{overpic}
\caption{\small The situation in the proof of \Lr{connection}.}
\label{disjoint}
\end{figure}

We can now conclude that $\sum_{v\in \partial^V P} |Q_v|\leq |\partial P|\leq \Delta |P|$, where $\Delta$ is the maximal degree of $G$. It follows that the number of $v\in \partial^V P$ such that $|Q_v|\geq |P|^{3/4}$ is at most $\Delta |P|^{1/4}$. By the isoperimetric inequality \eqref{isop} there is $c>0$ such that $|\partial^V P|\geq c |P|^{\frac{d-1}{d}}$, which implies the strict inequality $|\partial^V P|> \Delta |P|^{1/4}$ whenever $|P|$ is large enough. It is clear that $P_u$ has size $|P|-O(|P|^{3/4})$, because 
$|P\setminus P_u|\leq \Delta^t |Q_u|$. The proof is now complete.
\end{proof}

The boundary vertices satisfying the property of \Lr{connection} will be called \defi{good}, and the remaining ones will be called \defi{bad}. 

In order to generalise \Prr{o box} to all elements of $\mathcal{S}$, we will need the following definitions. Consider a quasi-transitive lattice $G$ in $\R^2$. Given two linearly independent vectors $z,w\in \R^2$ we write $B(z,w)$ for the box determined by $z$ and $w$. Given a side $s$ of $B(z,w)$, we write $B^s(z,w)$ for the box that is congruent to $B(z,w)$, and satisfies $B^s(z,w)\cap B(z,w)=s$. It is not hard to see that there are vectors $z_1,z_2,w_1,w_2$ such that the following hold:
\begin{itemize}
\item Both $z_1,w_1$ are parallel to $v_1$, and both $z_2,w_2$ are parallel to $u_2$.
\item For every side $s$ of $B(z_1,z_2)$, there are vertices $u\in B(z_1,z_2)$ and $v\in B^s(z_1,z_2)$ that can be connected with a path lying in $B(z_1,z_2)\cup B^s(z_1,z_2)$.
\item For every pair of vertices $u,v$ in $B(z_1,z_2)$, there is a path in $B(w_1,w_2)$ connecting $u$ to $v$.
\end{itemize}

We regard the tillings $\mathcal{T}_z$ and $\mathcal{T}_w$ of $\R^2$ by translates of $B(z_1,z_2)$ and $B(w_1,w_2)$, respectively, as graphs that are naturally isomorphic to $\Z^2$.

\begin{lemma}\label{o box site}
Consider a graph $G\in \mathcal{S}$. Then $\widetilde b_r(G) = b_r(G)$.
\end{lemma}
\begin{proof}
We handled $\Z^d$ and $\mathbb{T}^d$ above, so it only remains to handle \qtl s. We will focus on the case of site-\pint s with surface-to-volume ratio $r$ such that $b_r>1$, which is the hardest one.

Let $n\in \N$, $\epsilon>0$, $r>0$, and let $P\in I_{n,r,\epsilon}$. Recall that there is a $t>0$ such that the cycles in our basis of $\cc(G)$ have length at most $t$. Consider the set of boxes in $\mathcal{T}_w$ that either intersect the $2t$-neighbourhood of $P\cup \partial P$, or they share a common face with such a box. Let $B^t(P)$ be the smallest box with sides parallel to $w_1,w_2$ containing all these boxes. Write $s$ for a side of $B^t(P)$. Order the vertices of $B^t(P)$ arbitrarily. Among all vertices of $\partial^V P$ that are closest to $s$, there is one that is minimal. The minimal vertices associated to the sides of $B^t(P)$ are called \defi{extremal}. Each extremal vertex lies in some box of $\mathcal{T}_z$ that is called extremal as well (in case a vertex lies in more than one boxes of $T_z$, order the boxes arbitrarily and choose the minimal one). We define the shape of a site-\pint\ $P$ to be the tuple comprising the dimensions of the box $B^t(P)$, and the extremal vertices of $P\cup \partial P$. Using the polynomial growth of $G$, we immediately deduce that we have polynomially many choices $p(n)$ for the shape and auxiliary shape of any site-\pint\ $P$. We define $K$ as in the proof of \Prr{o box}.

By definition, all elements $P$ of $K$ have the same $B^t=B^t(P)$. It is not hard to see that at least one of the two dimensions of $B^t$ is $\Omega(\sqrt{n})$. Indeed, for every vertex $u$ of $G$ there is a disk of small enough radius $r_u>0$ containing no other vertex except for $u$. The translation invariance of $G$ implies that there are only finitely possibilities for $r_u$, hence $r=\inf_{u\in V} r_u>0$. It follows that $B^t$ has area $\Omega(n)$ because it contains $n$ disjoint disks of radius $r$. This implies that at least one of the two dimensions of $B^t$ is $\Omega(\sqrt{n})$. We can assume without loss of generality that the dimension parallel to $v_1$ has this property. 

We start with a $n\times n$ grid of copies $B_{i,j}$ of $B^t$. We place inside every $B_{i,j}$ a site-\pint\ $K_{i,j}\in K$. We write $S_k$ for the $k$th column of the grid. Similarly to the proof of \Prr{o box}, we move every column, except for the first one, in the direction parallel to $v_2$ in such a way that the `rightmost' extremal boxes of $S_k$ and the `leftmost' extremal boxes of $S_{k+1}$ can be connected in $\mathcal{T}_z$ by a straight path parallel to $v_1$.

For every pair $K_{i,j}$, $K_{i,j+1}$ of consecutive \pint s, there is an induced path in $G$ of bounded length connecting their `rightmost' and `leftmost' extremal vertices. We can further assume that the path lies in $B_{i,j}\cup B_{i,j+1}$ by our choice of $z_1,z_2,w_1,w_2$ and the definition of $B^t$. Indeed, if $\Pi=B_1,B_2,\ldots,B_l$ is a straight path in $\mathcal{T}_z$ connecting the `rightmost' and `leftmost' extremal boxes of $K_{i,j}$, $K_{i,j+1}$, respectively, we first connect all consecutive boxes $B_m,B_{m+1}$, $m=1,\ldots,l-1$ using paths $\Pi_m$ in $G$ lying in $B_m\cup B_{m+1}$. Then we connect the `rightmost' and `leftmost' endvertices of consecutive paths $\Pi_m$, $\Pi_{m+1}$, respectively, using paths lying in boxes congruent to $B(w_1,w_2)$ containing those endvertices. Finally, we connect the `rightmost' and `leftmost' of $K_{i,j}$, $K_{i,j+1}$ to $P_1$ and $P_{l-1}$ using paths lying in boxes congruent to $B(w_1,w_2)$. In this way we obtain a path that lies in $B_{i,j}\cup B_{i,j+1}$, because both $B_{i,j}$, $B_{i,j+1}$ contain a `layer' of boxes of $T_w$ surrounding $K_{i,j}$, $K_{i,j+1}$. The path is not necessarily disjoint from $K_{i,j}$, $K_{i,j+1}$ but it certainly contains a subpath that is disjoint from them, and connects two boundary vertices of both site-\pint s. We can choose the subpath to contain exactly two boundary vertices, one from each of the two site-\pint s.

Let $\mathcal{W}$ be the path connecting $K_{i,j}$, $K_{i,j+1}$, and let $u_1\in \partial^V K_{i,j}$, $u_2\in \partial^V K_{i,j+1}$ be the endvertices of $\mathcal{W}$. Adding $u_1,u_2$ to $K_{i,j}$, $K_{i,j+1}$ may result to much smaller site-\pint s. For this reason we need to find two good boundary vertices. Consider the vertices $x_1,x_2$ at distance $t$ from $u_1,u_2$, respectively, lying in $\mathcal{W}$. Write $Q_1,Q_2$ for the $(t-1)$-neighbourhood of $K_{i,j}$, $K_{i,j+1}$, respectively, and notice that both $\partial^V Q_1$, $\partial^V Q_2$ have distance $t$ from $\partial K_{i,j}$, $\partial K_{i,j+1}$, respectively. Furthermore, they coincide with the boundary of some site-\pint , i.e.\ the site-\pint\ of the finite connected component of their complement, and by Remark \ref{t/2} we can connect any pair of vertices of $\partial^V Q_i$, $i=1,2$ with a path lying in the $t/2$ neighbourhood of $\partial^V Q_i$, hence disjoint from $K_{i,j}$, $K_{i,j+1}$ and their boundaries. The isoperimetric inequality \eqref{isop} gives $\partial^V Q_i=\Omega(\sqrt{n})$. Moreover, for every $k>0$, the number of vertices of $\partial^V Q_i$ that can be connected to $x_i$ with a path of length at most $k$ lying in the $t/2$-neighbourhood of $\partial^V Q_i$ is $\Omega(k)$. On the other hand, \Lr{connection} implies that $O(n^{1/4})$ boundary vertices of either $K_{i,j},K_{i,j+1}$ are bad. Hence choosing $k=cn^{1/4}$ for some large enough constant $c>0$, we can find two good vertices $y_1,y_2$ in $\partial^V K_{i,j},\partial^V K_{i,j+1}$, that can be connected to $x_1,x_2$, respectively, in the following way: we first connect $y_i$ to some vertex of $\partial^V Q_i$ with a path of length $t$, and then we connect the latter vertex with a path of length $O(n^{1/4})$ lying in the $t/2$ neighbourhood of $\partial^V Q_i$. Taking the union of these two paths with the subpath of $\mathcal{W}$ connecting $x_1$ to $x_2$, we obtain a path of length $O(n^{1/4})$ connecting $y_1$ to $y_2$ that lies in $B_{i,j}\cup B_{i,j+1}$. We attach this path to our collection of site-\pint s. 

Consider now a site-\pint\ $K_{i,j}$ with $2\leq j\leq n-1$. Notice that exactly two paths emanate from $\partial K_{i,j}$, one of which has distance $O(n^{1/4})$ from the `rightmost' extremal vertex of $K_{i,j}$, and the other has distance $O(n^{1/4})$ from the `leftmost' extremal vertex of $K_{i,j}$. The two paths may possibly overlap, separating some vertices of $\partial K_{i,j}$ from infinity. However, the distance between the `rightmost' and the `leftmost' extremal vertex is $\Omega(\sqrt{n})$ because the dimension of $B_{i,j}$ that is parallel to $v_1$ is $\Omega(\sqrt{n})$. We can increase the value of $n$ if necessary to ensure that the paths do not overlap.

Moreover, we connect, as we may, the boundaries of consecutive site-\pint s $K_{i,1},K_{i+1,1}$ of the first column with induced paths of length $O(n)$ disjoint from any other site-\pint, only the endvertices of which intersect the boundary of $K_{i,1},K_{i+1,1}$.

In this way we obtain a graph $H$ containing all site-\pint s $K_{i,j}$. Consider the graph spanned by $H$, and let $Q$ be the site-\pint\ of this induced graph. We claim that $Q$ has size $n^3(1-o(1))$, and boundary size between $(r-\epsilon')|Q|$ and $(r+\epsilon')|Q|$, for some $\epsilon'=\epsilon+o(1)$. Indeed, for every site-\pint\ $K_{i,j}$ that does not lie in the first column, if $F_{i,j}\subset \partial K_{i,j}$ is the set of endvertices of the attached paths that emanate from $\partial K_{i,j}$, then the site-\pint\ of $K_{i,j}\cup F_{i,j}$ (which has size $n-O(n^{1/4})$) lies in the boundary of $Q$. Since we have $n^2-n$ such $K_{i,j}$, the claim follows readily. 

Each column $S_k$ has been moved at distance $O(kn)=O(n^2)=o(|Q|)$ from its original position. Hence $|B(Q)|=o(|Q|)$. It remains to show that the number of such $Q$ constructed is roughly $N^{n^2}$. Notice that we have not necessarily used the same paths to connect our \pint s, and so given such a $Q$, we cannot immediately recover all possible sequences $(K_{i,j})$ giving rise to $Q$. Our goal is to restrict to a suitable subfamily of $K^{n^2}$.

We claim that there are only subexponentially many in $n^3$ possibilities for the attached paths. Recall that all elements of $K$ have the same extremal vertices. The endvertices of every attached path have distance $O(n^{1/4})$ from a pair of extremal vertices. Using the polynomial growth of $G$, we conclude that there are only polynomially many choices in $n$ for each endvertex. Moreover, the paths connecting \pint s of the first column have length $O(n)$, and the remaining paths have length $O(n^{1/4})$. There are at most $\Delta^{O(n)}$ choices for each path connecting site-\pint s of the first column, and at most $\Delta^{O(n^{1/4})}$ choices for each of the remaining paths, because any path starting from a fixed vertex can be constructed sequentially, and there are at most $\Delta$ choices at each step. In total, there are $\Delta^{O(n^{9/4})}$ possibilities for the attached paths. This proves our claim. 

On the other hand, there are at least $N^{n^2}$ sequences $(K_{i,j})\in K^{n^2}$, hence for a subfamily of $K^{n^2}$ of size at least $N^{n^2}/{\Delta^{O(n^{9/4})}}$, we have used exactly the same paths. Let us restrict to that subfamily. Since we have fixed the paths connecting the elements of the subfamily, given some $Q$ constructed by the elements of that subfamily, we can now delete every vertex of the attached paths except for their endvertices to `almost' reconstruct all site-\pint s producing $Q$. To be more precise, if $(K_{i,j})$ and $(K'_{i,j})$ are two sequences producing the same $Q$, then the site-\pint s of $K_{i,j}\cup F_{i,j}$ and $K'_{i,j}\cup F_{i,j}$ coincide. By \Lr{local} below, if we fix a sequence $(K_{i,j})$ producing $Q$, then for each $i,j$ with $j>1$, there are subexponentially many in $n$ possible $K'_{i,j}$ as above. For each of the remaining $i,j$ there are at most exponentially many in $n$ $K'_{i,j}$ as above, since there are at most exponentially many site-\pint s in total. Therefore, each $Q$ can be constructed by subexponentially many in $n^3$ sequences. We can now deduce that we constructed roughly $N^{n^2}$ $Q$, and taking limits we obtain $\widetilde b_r = b_r$, as desired.
\end{proof}

We now prove the lemma mentioned in the above proof.

\begin{lemma}\label{local}
Let $G$ be a \qtl . Let $P$ be a site-\pint\ of size $n$ in $G$, and $F\subset \partial^V P$. Assume that the site-\pint\ of $P\cup F$ has size at least $n-O(n^{3/4})$. Then the number of site-\pint s $P'$ of size $n$ such that the site-\pint s of $P\cup F$ and $P'\cup F$ coincide, is $n^{O(n^{3/4})}$.
\end{lemma}
\begin{proof}
Consider a site-\pint\ $P'$ of size $n$ such that the site-\pint\ $X$ of $P\cup F$, and the site-\pint\ of $P'\cup F$ coincide. Let $k$ be the size of $P'\setminus X$. By our assumption $k=O(n^{3/4})$. Each connected component of $P'\setminus X$ is incident to some vertex of $P$, hence every vertex of $P'\setminus X$ has distance $O(n^{3/4})$ from $P$. By the polynomial growth of $G$, the number of vertices at distance $O(n^{3/4})$ from $P$ is at most 
$m$ for some $m=O(n^{3/2})$. There are $${m\choose k}\leq m^k =n^{O(n^{3/4})}$$ subsets of size $k$ containing vertices having distance at most $k$ from $P$. Therefore, there are $n^{O(n^{3/4})}$ site-\pint s $P'$ as above. 
\end{proof}

\begin{figure}[htbp] 
\centering
\noindent
\begin{overpic}[height=.5\linewidth, width=.5\linewidth]{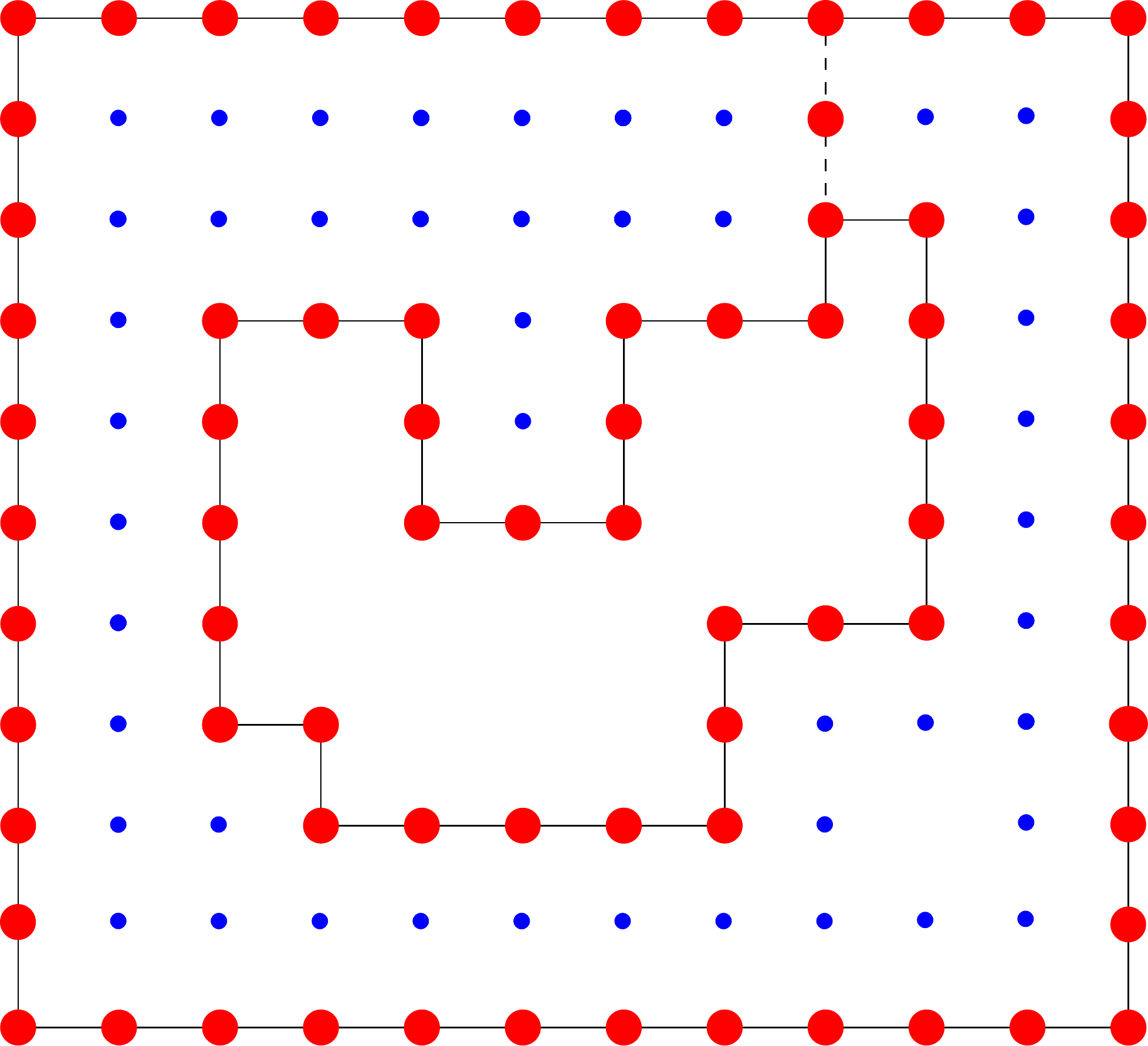}
\end{overpic}
\caption{\small The vertices of $Q$ are depicted with big disks, and the vertices of $\partial Q$ are depicted with smaller disks. The edges spanned by $P$ and $C$ are depicted in solid lines, while the edges of $\Pi$ are depicted in dashed lines.}
\label{ipint}
\end{figure}

We can now prove the main results of this section.

\begin{proof}[Proof of \Tr{dual site}]
Assume first that both $b_r$ and $b_{1/r}$ are not $0$.
We start by proving that \labtequ{star ineq}{$b_r^*\geq b_r$.} 
Combined with \eqref{b star} this will easily yield the desired equality.

Assume that $G$ is a \qtl . Let $n\in \N$, $r>0$, $\epsilon>0$, and choose $P\in I_{n,r,\epsilon}$. By \Prr{o box site}, we may assume that $P$ satisfies $|B(P)| = o(|P|)=o(n)$. Recall that $B(P)$ contains $P\cup \partial P$ in its interior. It is not hard to see that there is a cycle $C$ at bounded distance from $P$ that separates $B(P)$ from infinity, and has size $O(|B(P)|)$. Arguing as in the proof of \Lr{o box site} we find a good vertex $u\in \partial^V P$, and an induced path $\Pi$ connecting $u$ to $C$ that has size $O(n^{1/4})$, and does not contain any other vertex of $P\cup\partial P$. 

Our aim now is to find a suitable \ipint\ containing the site-\pint\ of $P\cup\{u\}$, which we denote by $X$. Since the cycle space of $G$ is generated by its triangles, $\partial^V P$ spans a cycle surrounding $P$ and the remaining boundary vertices. Hence $\partial^V P \setminus \{u\}$ spans a connected graph. The graph $\Gamma:=X\cup \Pi \cup C$ surrounds an open subset of the plane that contains $\partial^V P \setminus \{u\}$. Consider the connected component $Y$ of $\partial^V P \setminus \{u\}$ in this open set. Write $Q$ for the \ipint\ of $Y$, i.e.\ the boundary of the site-\pint\ of $Y$.

We claim that $Q$ contains $X$ and is contained in $X\cup \Pi \cup C$. To see that $Q$ contains $X$, notice that all vertices of $X$ are incident to $Y$ because $G$ is a triangulation, and lie in the external face of $Y$. Therefore, $X$ is contained in $Q$. Moreover, if $Q$ contains some vertex not in $X\cup \Pi \cup C$, then we can add this vertex to $Y$ to obtain an even larger connected graph. This contradiction shows that there is no such vertex and proves our claim. \medskip

We now consider the case where $G=\mathbb{T}^d$. We can let $C$ be the set of vertices in the boundary of $B(P)$, and $\Pi$ be a path of length $2$ connecting an extremal vertex of $B(P)$ to $P$. Let $Y$ be the subgraph of $G$ surrounded by $P\cup \Pi \cup \partial C$. It is clear that $Y$ is connected. Write $Q$ for the \ipint\ of $Y$. Every vertex of $P$ is incident to $Y$, and lies in the infinite component of $G\setminus Y$. Hence $P$ lies in $Q$. Furthermore, $Q$ contains only vertices of $X\cup \Pi \cup \partial C$.

In both cases, $Q$ has roughly $n$ vertices and surface-to-volume ratio between $(r-\epsilon')|Q|$ and $(r+\epsilon')|Q|$ for some $\epsilon'=\epsilon+o(1)$. Moreover, each $Q$ can be obtained from only subexponentially many $P$. This proves \eqref{star ineq}. Combining this with \eqref{b star}, we obtain the following:
$$b_r^*\geq b_r = (b^*_{1/r})^r \geq b_{1/r}^r = b^*_r,$$
where both inequalities coincide with \eqref{star ineq} and both equalities with \eqref{b star}. Thus we must have equality all along, and in particular $b_r = b_{1/r}^r$. 

If $b_r$ or $b_{1/r}$ is equal to $0$, then the same argument shows that 
both $b_r$ and $b_{1/r}$ are equal to $0$. This completes the proof.
\end{proof}

We now prove \Tr{dual bond}.

\begin{proof}[Proof of \Tr{dual bond}]
Assume first that both $b_r$ and $b_{1/r}$ are not $0$. Choose $P\in I_{n,r,\epsilon}$ such that $|B(P)| = o(|P|)=o(n)$. Consider a cycle $C$ as in the proof of \Tr{dual site} and connect $\partial P$ to $C$ with a path $\Pi$ of minimal length. Notice that $(\partial^E P)^*$ is a cycle, hence every $(\partial P \setminus E(\Pi))^*$ is a connected graph. 

Let $X$ be the connected component of $(\partial P \setminus E(\Pi))^*$ in $G^*$ that is surrounded by $P\cup \Pi \cup C$, and let $Q$ be the \pint\ of $X$.  Arguing as in the proof of \Tr{dual site}, we see that $P^*$ lies in $\partial Q$, $Q$ has size roughly $n$, and $(r-\epsilon')|Q|\leq |\partial Q|\leq (r+\epsilon')|Q|$ for some $\epsilon'=\epsilon+o(1)$.

Let $b^\bullet_r(G)$ be defined like $b_r(G)$ except that we now consider \ipint s. Thus we have 
\labtequ{star bull}{$b^*_r(G)= b^\bullet_r(G^*)$}
by the definitions. The above construction now yields the inequality
$b^\bullet_r(G) \geq b_r(G)$. 

Combining this with \eqref{b star}, which we rewrite using \eqref{star bull}, we obtain
$$b^\bullet_r(G) \geq b_r(G) = (b_{1/r}^\bullet(G^*))^r \geq b_{1/r}^\bullet(G^*)^r = b^\bullet_r(G),$$
as above, and again equality holds all along. In particular,\\
$b_r(G) = b_{1/r}^\bullet(G^*)^r = (b_{1/r}^*(G))^r$.

If $b_r$ or $b_{1/r}$ is equal to $0$, then the same argument shows that 
both $b_r$ and $b_{1/r}$ are equal to $0$. This completes the proof.
\end{proof}

The arguments in the proofs of \Lr{o box} and \Tr{o box site} can be used to prove \Lr{pint mpint}.

\begin{proof}[Proof of \Lr{pint mpint}]
The inequality $b_r^\circ \leq b_r^\odot$ is obvious.

For the reverse inequality we will focus on the case of site-\pint s. We will construct an array of a certain number of boxes of possibly different sizes, then place the component site-\pint s of an arbitrary site-\mpint\ inside the boxes, and connect them with short paths to obtain a new site-\pint . 

We claim that the number of choices for the shapes of the components of any site-\mpint\ of size $n$ grows subexponentially in $n$. Indeed, the number of choices for the shape of any site-\pint\ grows polynomially in its size. \Tr{HRthm} shows that there are most $s^{\sqrt{n}}$ choices for the component sizes of any site-\mpint\ of size $n$, where $s>0$ is a constant. Hence it suffices to show that a site-\mpint\ of size $n$ comprises $O(\sqrt{n})$ site-\pint s.

Let $X=(\ldots,x_{-1},x_0=o,x_1,\ldots)$ be a quasi-geodesic in $G$ containing $o$ and let $X^+=(x_0,x_1,\ldots)$ be the one of the two $1$-way infinite subpaths of $X$ starting from $o$. Consider a site-\mpint\ $P$ of size $n$. As remarked in the proof of \Prr{br ineq}, $P$ contains at least one of the first $ln+1$ vertices of $X^+$. We enumerate the component site-\pint s $P_1,P_2,\ldots,P_k$ of $P$ according to the first vertex of $X^+$ that they contain. As the $P_i$'s are disjoint, we have $l_i<l_{i+1}$, where $l_i$ is the index of the first vertex of $X^+$ that $P_i$ contains. Since $l_1\geq 0$, we deduce that $l_i\geq i-1$ for every $i$. Hence, we obtain
$$|P_i|\geq (i-1)/l$$
for every $i=1,2,\ldots,k$, which implies that 
$$n=\sum_{i=1}^k |P_i|\geq \sum_{i=1}^k (i-1)/l=\frac{k(k-1)}{2l}.$$
The latter implies that $k=O(\sqrt{n})$, hence there are $(sn)^{O(\sqrt{n})}$ choices for the shapes of the components site-\pint s of any site-\mpint\ of size $n$.

We can now restrict to a subfamily $K\subset MI_{n,r,\epsilon}$ of size at least $$N:=\dfrac{c_{n,r,\epsilon}}{(sn)^{O(\sqrt{n})}}$$ such that all site-\mpint s of $K$ have the same component sizes, say $\{n_1,n_2,\ldots,n_k\}$, and corresponding component site-\pint s have the same shape. Let $B_1,B_2,\ldots,B_k$ be the boxes of the component site-\pint s. Instead of a grid, we construct an array by placing the above $k$ boxes next to each other. Given an element of $K$, we place its component site-\pint s in their boxes. After moving the boxes, if necessary, we connect them with short paths, as described in the proof of \Lr{o box site}. Arguing as in the proof of \Lr{o box site} we obtain $b_r^\circ \geq b_r^\odot$, as desired.
\end{proof}

Since $\Pr_{p_c}(|S_o|=n)$ does not decay exponentially in $n$, we conclude

\begin{corollary} \label{cor no ed}
Consider site percolation on a graph $G\in \mathcal{S}$ satisfying  \eqref{triangulation}. Then $\Pr_{1-p_c}(|S_o|=n)$ does not decay exponentially in $n$.
\end{corollary}
\begin{proof}
Notice that $r(1-p_c)=1/r(p_c)$. The fact that $\Pr_{p_c}(|S_o|=n)$ does not decay exponentially in $n$ implies that $b_{r(p_c)}=f(r(p_c))$. \Tr{dual site} shows that 
$$b_{r(1-p_c)}=b_{r(p_c)}^{1/r(p_c)}=f(r(p_c))^{1/r(p_c)}=f(r(1-p_c)).$$ Using \Tr{S_o} we conclude that $\Pr_{1-p_c}(|S_o|=n)$ does not decay exponentially in $n$.
\end{proof}

\section{Continuity} \label{cont dec} 

In this section we study the analytical properties of $b_r$. To avoid repeating the arguments in the proof of \Lr{o box site} and considering cases according to whether we study \pint s or site-\pint s, we will prove the results for \pint s in $\Z^d$ and $\mathbb{T}^d$. 
\

We first prove that the $\limsup$ in the definition of $b_r$ can be replaced by $\lim$. 
\begin{proposition} \label{limit}
Let $G\in \mathcal{S}$. Then for every $r$ such that $b_r>1$ and for all but countably many $\epsilon>0$ the limit $\lim_{n\to\infty} c_{n,r,\epsilon}(G)^{1/n}$ exists. 
\end{proposition}
\begin{proof} 
We will first show that $$\limsup_{n\to \infty} c_{n,r,\epsilon}^{1/n} = \liminf_{n\to\infty} c_{n,r,\epsilon}^{1/n}$$ holds for any $\epsilon>0$ at which the function $\liminf_{n\to \infty} c_{n,r,\epsilon}^{1/n}$ is continuous. Since $\liminf_{n\to\infty} c_{n,r,\epsilon}^{1/n}$ is an increasing function of $\epsilon$, its points of discontinuity are countably many \cite{BabyRudin}. 

Let $\epsilon$ be a point of continuity of $\liminf_{n\to\infty} c_{n,r,\epsilon}^{1/n}$ and $n\in \N$. By combining elements of $I_{n,r,\epsilon}$ we will construct \pint s of arbitrarily large size and surface-to-volume ratio between $r-\epsilon'$ and $r+\epsilon'$ for some $\epsilon\leq \epsilon'=\epsilon+o(1)$. Let $0\leq s\leq n+3$ be an integer. We repeat the idea of \Prr{o box}, but instead of a grid, we construct an array of $m$ boxes for some $m>0$. We place inside each box an element of $I_{n,r,\epsilon}$ and after moving the boxes, if necessary, we connect consecutive \pint s using paths of length $4$, similarly to the proof of \Prr{o box}. We also attach a path of length $s+4$, that is incident to the last \pint\ and disjoint from any of the previous \pint s. In this way we produce an element $Q$ of $I_{k,r,\epsilon'}$, where $\epsilon'=\epsilon+o(1)$ and $k$ is any integer of the form $k=m(n+4)+s$. There are roughly $c_{n,r,\epsilon}^m$ choices for $Q$. Since $s$ ranges between $0$ and $n+3$, for every fixed $n$, all but finitely many $k$ can be written in this form for some $m\geq 1$. Taking the $k$th root and then the limit as $m\to \infty$ we conclude that $\liminf_{k\to\infty} c_{k,r,\epsilon'}^{1/k}\geq c_{n,r,\epsilon}^{1/(n+4)}$. Letting $n\to \infty$ we obtain $\liminf_{n\to \infty} c_{n,r,\epsilon}^{1/n} \geq \limsup_{n\to\infty} c_{n,r,\epsilon}^{1/n}$. The above inequality follows from the fact that $\epsilon$ is a point of continuity of $\liminf_{n\to\infty} c_{n,r,\epsilon}^{1/n}$. Hence $\liminf_{n\to \infty} c_{n,r,\epsilon}^{1/n} = \limsup_{n\to\infty} c_{n,r,\epsilon}^{1/n}$, as desired
\end{proof}

The following proposition follows directly from the definition of $b_r$:
\begin{proposition}\label{upper semi}
Let $G\in \mathcal{S}$. Then $b_r(G)$ is an upper-semicontinuous function of $r$.
\end{proposition}
\begin{proof} 
Let $\epsilon>0$ and $0<\delta<\epsilon/2$. Then for every $r>0$ and for every $s$ with $|r-s|<\epsilon/2$, the interval $(s-\delta,s+\delta)$ is contained in $(r-\epsilon,r+\epsilon)$, 
and the site-\pint s $P$ with $|\partial P|/|P| \in (s-\delta,s+\delta)$ are counted in the set of those site-\pint s with $|\partial P|/|P| \in (r-\epsilon,r+\epsilon)$ as well. Hence, $\limsup_{n\to\infty} c_{n,r,\epsilon}^{1/n}\geq \limsup_{n\to \infty} c_{n,s,\delta}^{1/n}$. Taking limits as $\delta\to 0$, $s\to r$ 
and finally $\epsilon\to 0$, we obtain $b_r \geq \limsup_{s\to r} b_s$. The latter shows that $b_r$ is an upper-semicontinuous function of $r$.
\end{proof}

Next, we prove that $b_r$ is a log-concave function of $r$:

\begin{proposition}\label{concave}
Let $G\in \mathcal{S}$. Then for any $t\in [0,1]$ and any $r,s$ such that $b_r(G),b_s(G)>1$, we have $b_{tr+(1-t)s}(G) \geq b_r(G)^{t} b_s(G)^{1-t}$.
\end{proposition}
\begin{proof} 
Pick an $\epsilon$ such that both $\lim_{n\to\infty} c_{n,r,\epsilon}^{1/n}$ and $\lim_{n\to\infty} c_{n,s,\epsilon}^{1/n}$ exist. Let $(p_m/q_m)$ be a sequence of rational numbers converging to $t$ such that $q_m\to\infty$. Consider subfamilies $K$, $K'$ of $I_{p_m,r,\epsilon}$ and $I_{q_m-p_m,s,\epsilon}$, where the elements of both $K$ and $K'$ have the same shape (as defined in the proof of \Prr{o box}), and $|K|\geq c_{p_m,r,\epsilon}/P(p_m)$, $|K'|\geq c_{q_m-p_m,s,\epsilon}/P(q_m-p_m)$ for some polynomial $P(x)$. Note that the elements of $K$ and $K'$ share the same boxes $B$ and $B'$, respectively. Place two \pint s, one from $K$ and another from $K'$, in an array of two boxes parallel to $B$ and $B'$, and move the boxes, if necessary, in order to connect the \pint s with short paths. In this way we obtain an \pint\ $Q$ of size roughly $q_m$ and surface-to-volume ratio roughly $tr+(1-t)s$. Notice that we have at least $c_{p_m,r,\epsilon} c_{q_m-p_m,s,\epsilon}/\big(P(p_m) P(q_m-p_m)\big)$ choices for $Q$. Taking the $k$th root of the latter expression, where $k=|Q|$, and letting $m\to \infty$ gives
$$\lim_{n\to \infty} c_{n,r,\epsilon}^{t/n} \lim_{n\to \infty} c_{n,s,\epsilon}^{(1-t)/n}.$$ Letting $\epsilon\to 0$ along a sequence of points such that both $\lim_{n\to\infty} c_{n,r,\epsilon}^{1/n}$ and $\lim_{n\to\infty} c_{n,s,\epsilon}^{1/n}$ exist, we obtain $b_{tr+(1-t)s} \geq b_r^t b_s^{1-t}$ as desired.
\end{proof}

We expect \Prr{concave}, and as a result \Tr{cont} below, to hold in much grater generality than $G\in \mathcal{S}$, namely for all 1-ended \Cg s. In order to be able to put several \pint s close to each other to connect them with short paths as in the above proof, it could be handy to use \cite[Lemma~6]{BanStTim}.

\medskip
Let $I$ be the closure of the set of $r$ such that $b_r>1$. \Prr{concave}, combined with \Prr{upper semi}, easily imply
\begin{theorem}\label{cont}
Let $G\in \mathcal{S}$. Then $b_r(G)$ is a continuous function of $r$ on $I$. 
\end{theorem}
\begin{proof}
By \Prr{concave}, $I$ is an interval, and the only possible $r\in I$ such that $b_r=1$, are its endpoints. For every $r$ in $I$, we have $\limsup_{s \to r} b_s \leq b_r$ by \Prr{upper semi}. Using \Prr{concave} for $t=1/2$ we obtain $\liminf_{s\to r} b_{(r+s)/2} \geq \sqrt{b_r \liminf_{s \to r} b_s}$ for every $r$ such that $b_r>1$. This immediately implies that $\liminf_{s\to r} b_s\geq b_r$ and thus $\lim_{s\to r} b_s = b_r$. 

On the other hand, if $b_r=1$ for some of the endpoints of $I$, then
\Prr{upper semi} and the fact that $b_s>1$ for $s$ in the interior of $I$, give that $$\mathop{\lim_{s\to r}}_{s\in I} b_s= 1.$$ Therefore, $b_r$ is a continuous function on $I$.
\end{proof}

Having proved that $b_r$ is a continuous function, the next natural question is whether it is differentiable. It turns out that this holds everywhere except, perhaps, on a countable set.

\begin{corollary} \label{differentiable}
Let $G\in \mathcal{S}$. Then $b_r(G)$ is differentiable for all but countably many $r$.
\end{corollary}
\begin{proof}
By \Prr{concave}, $\log b_r$ is a concave function, hence differentiable everywhere except for a countable set \cite{RockConvex}. It follows immediately that this holds for $b_r$ as well.
\end{proof}

\section{An analogue of the Cheeger constant for \pint s} \label{analogue}

In this section we define a variant $I(G)$ of the Cheeger constant as the infimal surface-to-volume ratio over all interfaces rather than arbitrary finite subgraphs of $G$. 
\medskip

Given a graph $G$ and a finite set of vertices $S$, we define the edge boundary $\partial_E S$ of $S$ to be the set of edges of $G$ with one endvertex in $S$ and one not in $S$. The vertex boundary $\partial_V S$ of $S$ is defined to be the set of vertices in $V\setminus S$ that have a neighbour in $S$. The edge Cheeger constant of $G$ is defined as $h_E(G)=\inf_S \frac{|\partial_E S|}{|S|}$, where the infimum is taken over all finite sets $S$ of vertices. The vertex Cheeger constant of $G$ is defined as $h_V(G)=\inf_S \frac{|\partial_V S|}{|S|}$, where the infimum is taken again over all finite sets of vertices $S$. In \cite{BeSchrPer} Benjamini \& Schramm proved that for all non-amenable graphs $G$, $p_c(G)\leq\frac{1}{h_E(G)+1}$ and $\pcs(G)\leq\frac{1}{h_V(G)+1}$.

Analogously to the Cheeger constant, we define $I_E(G):=\inf_{P} \frac{|\partial P|}{|P|}$, $I_V(G):=\inf_{Q} \frac{|\partial Q|}{|Q|}$ where the infimum ranges over all \pint s $P$ and site-\pint s $Q$, respectively. It is possible that $I_E(G)>0$ and $I_V(G)>0$ while $h_E(G)=0$ and $h_V(G)=0$, e.g.\ for Cayley graphs of amenable finitely presented groups.

In the following theorem we prove an upper bound on $p_c$ which is reminiscent of the result of Benjamini \& Schramm. Our result applies to graphs for which \labtequ{weak isop}{there exists a function $g(n)$ of sub-exponential growth with the property that for every connected subgraph $H$ of $G$ with $n$ vertices, we have $|\partial^V H| \leq g(n)$.}
This assumption is satisfied, for example, by the class of graphs $\mathcal{S}$  that we were working in the previous section for $g(n)=n^{\frac{d-1}{d}}$. It has the following implication. Consider a vertex $o$ of $G$ and let $X$ be an infinite geodesic starting from $o$. Let also $P$ be a site-\pint\ of $o$ such that $|\partial P|=n$. Then \labtequ{iso}{$P$ contains one of the first $g(n)$ vertices of $X$.}

\begin{theorem}\label{cheeger-like}
Let $G$ be an $1$-ended, $2$-connected graph $G$ satisfying \eqref{weak isop}. Then for bond percolation on $G$ we have $p_c(G)\leq \frac{1}{I_E(G)+1}$, and for site percolation on $G$ we have $\pcs(G)\leq \frac{1}{I_V(G)+1}$.
\end{theorem}
\begin{proof}
We will prove the assertion for bond percolation. The case of site percolation is similar.

If $I_E(G)=0$ there is nothing to prove, so let us assume that $I_E(G)>0$. Let $q=\frac{1}{I_E(G)+1}$ and $r=1/{I_E(G)}$. We claim that for every $p>q$ the expected number of occurring 
\mpint s of boundary size $n$ decays exponentially in $n$. Indeed, let $\mathcal{P}_{n}$ be the (random) number of occurring \mpint s $P$ with $|\partial P|=n$. Notice that 
\labtequ{bounded ratio}{$|P|\leq rn$ for every \mpint\ $P$ with $|\partial P|=n$,} hence $\mathcal{P}_{n}\leq G(n)$, where $G(n)=e^{o(n)}$, because at most $g(k)$ \pint s of size smaller that $k$ can occur in any percolation instance $\omega$. Indeed, we can choose $G(n)=p(rn)\max\{g(n_1)\ldots g(n_i)\}$, where the maximum ranges over all possible partitions of $rn$. It is not hard to see that $G(n)=e^{o(n)}$.

Now for any $p>q$ and every $m\leq rn$ we have
\begin{gather*}
p^m (1-p)^n=q^m (1-q)^n (p/q)^m \Big(\dfrac{1-p}{1-q}\Big)^n \leq 
\\ q^m (1-q)^n (p/q)^{rn} \Big(\dfrac{1-p}{1-q}\Big)^n =q^m (1-q)^n \big(f(r)\big)^n p^{rn} (1-p)^n,
\end{gather*}
which implies that
\begin{equation}\label{smaller than 1}
\Ex_p(\mathcal{P}_n)\leq \Ex_q(\mathcal{P}_n) \big(f(r)\big)^n p^{rn} (1-p)^n \leq G(n) \big(f(r)\big)^n p^{rn} (1-p)^n.
\end{equation}
Notice that $f(r) p^r (1-p)<1$, which proves our claim.

The fact that $\Ex_p(\mathcal{P}_n)$ decays exponentially in $n$ implies that the sum $$\sum_{n=1}^\infty\sum_{\substack{P\in \MS \\ |\partial P|=n}} (-1)^{c(P)+1} \Pr_p(\text{$P$ occurs})$$  converges absolutely, where $\MS$ is the set of \mpint s, and $c(P)$ is the number of component \pint s of $P$. The inclusion-exclusion principle then implies that we can express $1-\theta$ as 
$$1-\theta(p)=\sum_{n=1}^\infty \sum_{\substack{P\in \MS \\ |\partial P|=n}} (-1)^{c(P)+1} \Pr_p(\text{$P$ occurs})$$ for
every $p>\frac{1}{I_E(G)+1}$. We claim that the latter sum is an analytic function on $(\frac{1}{I_E(G)+1},1]$. To prove this, we will use  \cite[Corollary 4.14]{analyticity}, which states that any such expression is analytic provided the following two conditions are satisfied. Firstly, the event $\{\text{$P$ occurs}\}$ should depend on $\Theta(|\partial P|)$ edges for every $P\in \MS$, and secondly,  for every $q\in (\frac{1}{I_E(G)+1},1)$, there exists $0<c=c(q)<1$ such that $$\sum_{\substack{P\in \MS \\ |\partial P|=n}} (-1)^{c(P)+1} \Pr_p(\text{$P$ occurs})=O(c^n)$$
for every $p\in (q,1)$. 
In our case, the first condition follows from \eqref{bounded ratio}. The second condition follows from \eqref{smaller than 1} and the fact that for every $q\in (\frac{1}{I_E(G)+1},1)$, there exists $0<c=c(q)<1$ such that $f(r) p^r (1-p)\leq c$ for every $p\in (q,1)$. We can thus deduce that $\theta$ is analytic on the interval $(\frac{1}{I_E(G)+1},1]$. \mymargin{If we use Corollary 4.13 we have to say what it says. Or is there something simpler we can use? More details are needed here. Ans: Done} Since $\theta$ is not analytic at $p_c$, it follows that $p_c(G)\leq \frac{1}{I_E(G)+1}$. 
\end{proof}
\begin{remark}
There is an alternative way to prove \Tr{cheeger-like}, without showing that $\theta$ is analytic, by means of \eqref{bounded ratio} and a Peierls-type of argument. However, this alternative way is not simpler. We emphasize that the exponential decay of $\Ex_p(\mathcal{P}_n)$ does not imply, in general, that $p>p_c$, as it can happen that $\Ex_p(\mathcal{P}_n)$ decays exponentially for some $p<p_c$.
\end{remark}

Clearly $I_E(G)\geq h_E(G)$ and $I_V(G)\geq h_V(G)$, hence the above theorem gives an alternative proof of the result of Benjamini \& Schramm mentioned above. 

\begin{question}
For which (transitive) non-amenable graphs does the strict inequality $I_E(G)> h_E(G)$ ($I_V(G)> h_V(G)$) hold?
\end{question}

It has been recently proved   that $I_V(G)> h_V(G)$ holds for the $d$-regular triangulations and quadrangulations of the hyperbolic plane \cite{SitePercoPlane}.

\section{Strict inequalities}

As we already observed, interfaces are a special species of lattice animals. 
In this section, we will show that there are exponentially fewer interfaces than lattice animals for a large class of graphs. Building on this result, we will then show that the strict inequality $\dot{p}_c<\frac{1}{h_V+1}$ holds for a large class of non-amenable graphs. It will be more convenient for us to work with the site variants of interfaces and lattice animals.

We start by introducing some notation regarding site-interfaces on general graphs. Consider an $1$-ended, $2$-connected, bounded degree graph $G$. Furthermore, assume that both \eqref{finitely presented} and \eqref{weak isop} are satisfied. For example, we can take $G$ to be a Cayley graph of an $1$-ended, finitely presented group. We fix a basis of the cycle space consisting of cycles of bounded length and we consider the family of site-interfaces associated to this basis. As we are working with a general graph which might not have any symmetries, the number of site interfaces might depend on the base vertex $o$. For this reason,
given a vertex $o$ of $G$, we define $\dot{b}_{n,o}=\dot{b}_{n,o}(G)$ as the number of site-interfaces $P$ of $o$ with $|P|=n$ and we let $\dot{b}_n=\dot{b}_n(G):=\sup_{o\in V(G)} \dot{b}_{n,o}$. We now define $\dot{b}=\dot{b}(G):=\limsup_{n\to \infty} \dot{b}^{1/n}_n$. 

Next, we formally define lattice site-animals. A \defi{lattice site-animal} $S$ is a set of vertices of $G$ that spans a connected graph. As above, given a vertex $o$ of $G$, we define $\dot{a}_{n,o}=\dot{a}_{n,o}(G)$ as the number of lattice site-animals containing $o$ with $n$ vertices and we let $\dot{a}_n=\dot{a}_n(G):=\sup_{o\in V(G)} \dot{a}_{n,o}$. We now define $\dot{a}=\dot{a}(G):=\limsup_{n\to \infty} \dot{a}^{1/n}_n$. 

For every vertex $v$ of $G$, we let $\cp_v$ be the set of basic cycles containing $v$ and we let $C_v$ be the union of all cycles in $\cp_v$. Using an argument similar to that in the proof of Kesten's pattern theorem for self-avoiding walks \cite{KestenI}, we will prove that 

\begin{theorem}\label{strict}
Let $G$ be an $1$-ended, $2$-connected, bounded degree graph $G$ satisfying \eqref{finitely presented} and \eqref{weak isop}. 
Then $\dot{b}(G)<\dot{a}(G)$.
\end{theorem}
\begin{proof}
Consider a site-\pint\ $P$ of $o$ of size $n$. We will introduce a local operation that turns $P$ into a lattice site-animal that is not a site-\pint. To this end, let $l$ be the length of the longest cycle in $\cp$. Consider a set of vertices $S$ of $P$ with the property that any two vertices in $S$ are at distance at least $2l+1$ apart, and $S$ is a maximal set with respect to this property. It is not hard to see that the graphs $C_v$, $v\in S$ are pairwise disjoint. Moreover, we have $|S|\geq n/d^{4l+2}$, where $d$ is the maximal degree of $G$, because the maximality of $S$ implies that any vertex of $P$ is at distance at most $4l+2$ from some vertex in $S$.

Let $0<\varepsilon<1/d^{4l+2}$. For each subset $T$ of $S$ of cardinality $m=\left \lfloor{\varepsilon n}\right \rfloor$, we let $P(T)=(\cup_{v\in T}C_v)\cup P$.
Recall that by \ref{pint b}, a vertex $v$ is included in $P$ only if there is a \cp-path in $G\setminus \partial P$ connecting $v$ to $\ard{\partial P}{D}$, hence $C_v$ is not contained entirely in $P$. It follows that among the vertices in $S$, only those in $T$ have the property that $C_v$ is contained in $P(T)$. Thus the graphs $P(T)$ are pairwise distinct and there are ${|S| \choose m}$ of them.

Our aim now is to find a lower bound for the cardinality of the set $$Y_n:=\bigcup_{P} \{P(T) \mid T\subset S, |T|=m\}.$$ To this end, we need to count the number of pre-images of each $P(T)$. Let $P'$ be a site-\pint\ of size $n$ and $T'$ be a set of vertices of $P'$ such that $P'(T')=P(T)$. Notice that any vertex of $T'$ needs to be at distance at most $2l$ from some vertex in $T$. Hence $P'$ coincides with $P$ on vertices at distance at most $3l$ from some vertex in $T$. Thus there are at most $2^{d^{3l}m}$ possibilities for $P'$ and $T'$.

Overall, we obtain $$|Y_n|\geq \dfrac{{r \choose m}}{2^{d^{3l}m}}\dot{b}_{n,o},$$
where $r=\left \lceil{n/d^{4l+2}}\right \rceil$. Using Stirling's approximation $N!=N^Ne^{-N+o(N)}$, we obtain 
$${r \choose m}=\dfrac{r^r e^{o(n)}}{m^m(r-m)^{r-m}}\geq \Big(\dfrac{r}{m}\Big)^m e^{o(n)}=\Big(\dfrac{1}{\varepsilon d^{4l+2}}\Big)^{\varepsilon n+o(n)}.$$

We can now deduce that
$$|Y_n|\geq \Big(\dfrac{1}{\varepsilon C}\Big)^{\varepsilon n+o(n)}\dot{b}_{n,o},$$
where $C=2^{d^{3l}}d^{4l+2}$. Choosing $\varepsilon=\big(C(2\dot{b})^{d^l}\big)^{-1}$ we obtain that $$|Y_n|\geq (2\dot{b})^{d^l \varepsilon n+o(n)}\dot{b}_{n,o}.$$

We claim that $$\dot{a}_{k,v}\geq (2\dot{b})^{d^l \varepsilon n+o(n)}\dot{b}_{n,o},$$ for some $n\leq k \leq n+d^l m$ and some vertex $v$. Indeed, the size of each element of $Y_n$ varies between $n$ and $n+d^l m$, so for some $n\leq k \leq n+d^l m$, the number of elements of $Y_n$ of size $k$ is at least $(2\dot{b})^{d^l \varepsilon n+o(n)}\dot{b}_{n,o}$. Now \eqref{iso} implies that for some vertex $v$, the number of elements of $Y_n$ of size $k$ that contain $v$ is at least $(2\dot{b})^{d^l \varepsilon n+o(n)}\dot{b}_{n,o}$, which proves the claim. 

Choosing some $o$ that maximizes $\dot{b}_{n,o}$, we obtain that 
$$\dot{a}_{k,v}\geq 2^{d^l \varepsilon n} \dot{b}^{(1+d^l \varepsilon)n+o(n)}$$ for some
$n\leq k \leq n+d^l m$ and some vertex $v=v(n)$. Taking the $k$th root and letting $k$ tend to infinity, we obtain that $$\dot{a}\geq 2^{\frac{d^l\varepsilon}{1+d^l\varepsilon}}\dot{b}$$ because $k\leq n+d^l m\leq (1+d^l \varepsilon)n$. This proves the desired result. 
\end{proof}

We will now assume that $G$ is a non-amenable graph. As already mentioned, it is a well-known result of Benjamin \& Schramm that $\dot{p}_c(G)\leq 1/(h_V+1)$ \cite{BeSchrPer}. We prove that this inequality is in fact strict.

\begin{theorem} \label{strict BS}
Let $G$ be an $1$-ended, $2$-connected, bounded degree, non-amenable graph $G$ satisfying \eqref{finitely presented}.
Then $\dot{p}_c(G)<\frac{1}{h_V(G)+1}$.
\end{theorem}
\begin{proof}
We will prove that for the values of $p$ in a neighbourhood of $1/(h_V(G)+1)$, the expected number of occurring site-\pint s of size $n$ decays exponentially in $n$.  This easily implies the theorem as for example, we can argue as in the proof of \Tr{cheeger-like} to deduce that $\theta$ is analytic in a neighbourhood of $1/(h_V(G)+1)$.

Consider some $\delta\in (0,1/h_V)$ and let $I_{n,\delta}$ be the set of site-\pint s $P$ such that $|\partial P|=n$ and $|P|\geq (1/h_V-\delta) n$. We will show that for a small enough $\delta$, the expected number of occurring elements of $I_{n,\delta}$ decays exponentially. Arguing as in the proof of \Tr{strict} and using the notation introduced there, we can associate to $I_{n,\delta}$ a set $Y_{n,\delta}$ of connected subgraphs such that 
$$|Y_{n,\delta}|\geq\Big(\dfrac{1}{\varepsilon C}\Big)^{\varepsilon(1/h_V-\delta) n+o(n)}|I_{n,\delta}|.$$
Given $P\in I_{n,\delta}$, $Q\in Y_{n,\delta}$ with $P$ being a pre-image of $Q$, we would like to show that $\Pr_p(P \text{ occurs})$ is not much larger than $\Pr_p(Q \text{ is open and } \partial_V Q \text{ is closed})$ (the latter is the probability that $Q$ is a cluster). In order to do this, we need to estimate the size of $\partial_V Q$. To this end, let $D$ be the union of all the finite components of $G\setminus \partial P$. By \ref{pint x} and \ref{pint b}, $P$ is contained in $D$ and $\partial_V D=\partial P$. Notice that 
$$n=|\partial P|\geq h_V\big(|P|+\big(|D|-|P|\big)\big)\geq n-h_V\delta n+h\big(|D|-|P|\big),$$ which implies that 
\begin{equation}\label{delta}
|D|-|P|\leq \delta n.
\end{equation}

Since $Q$ can be obtained from $P$ by adding at most $d^l$ vertices at $\left\lfloor{\varepsilon |P|}\right \rfloor$ places, we obtain that $|Q| \leq (1+d^l \varepsilon)|P|$. Moreover, by considering the size of the neighbourhood of $Q\setminus P$, we see that $$|\partial_V Q| \leq |\partial_V P|+d^{l+1}\varepsilon|P|\leq |\partial P|+|D|-|P|+d^{l+1}\varepsilon|P|,$$ where the term $|D|-|P|$ comes from the fact that $\partial_V P\setminus \partial P\subset D\setminus P$.
Using \eqref{delta} and the last inequality, we find that $$\Pr_p(P \text{ occurs})\leq \big(p(1-p)\big)^{-K(\delta+\varepsilon)n}\Pr_p(Q \text{ is open and } \partial_V Q \text{ is closed})$$ for a certain constant $K>0$. We can now choose $\delta=\varepsilon$ small enough that for every $p\in [h_V/2,(h_V+1)/2]$ we have
$$\sum_{P\in I_{n,\delta}}\Pr_p(P \text{ occurs})\leq 2^{-\varepsilon n} \sum_{Q\in Y_{n,\delta}} \Pr_p(Q \text{ is open and } \partial_V Q \text{ is closed}).$$ Now the non-amenability of $G$ implies that each $Q\in Y_{n,\delta}$ contains one of the first $2^{o(n)}$ vertices of $X$, hence
$\sum_{Q\in Y_{n,\delta}} \Pr_p(Q \text{ is open and } \partial_V Q \text{ is closed})=2^{o(n)}$ which in turn implies that $\sum_{P\in I_{n,\delta}}\Pr_p(P \text{ occurs})$ decays exponentially in $n$.

On the other hand, it follow from \Lr{large deviation} that for every $x>0$, the expected number of occurring interfaces $P$ with $|\partial P| =n$ and $|P|< \big(\frac{p}{1-p}-x\big)n$ decays exponentially in $n$. Hence for every $p>q\coloneqq\frac{1-h\delta}{h+1-h\delta}$, the expected number of occurring interfaces $P$ with $|\partial P| =n$ and $$|P|< \left(\frac{1}{h}-\delta\right) n=\frac{qn}{1-q}$$ under $\Pr_p$, decays exponentially in $n$, because $\frac{q}{1-q}<\frac{p}{1-p}$. Since we have $q<1/(h+1)$, the latter exponential decay holds at a neighbourhood of $1/(h+1)$. This completes the proof.
\end{proof}

\section{Conclusion}

In this paper we obtained basic properties of the function $b_r(G)$, and connected it to percolation theory and the enumeration of lattice animals. Many questions about $b_r(G)$ are left open, of which we mention just a few. We remarked that $\max_r b_r(G)$ is interesting, as it coincides with $b(G)$, which lower bounds the growth rate $a(G)$ of lattice animals. We expect that this maximum is attained at a single point $r= r_{max}$. What can be said about $r_{max}$? Is it always greater than $r(p_c)$? Is their ratio, or some other expression, independent of the lattice $G$ once the dimension is fixed?

We observed that $b_r$ is a continuous, almost everywhere differentiable function of $r$. Are stronger smoothness conditions satisfied? Is it smooth/analytic at every $r\neq r(p_c), r(1-p_c)$?

\section{Appendix: Continuity of the decay exponents} \label{cont dec exp}

In this appendix we prove that the rate of exponential decay $$c(p):= \lim_{n\to\infty} \Pr_p(|C_o|=n)^{1/n}$$ of the cluster size distribution ---which is known to exist for every $p\in (0,1)$ \cite{BanStTim,Grimmett}--- is a continuous function of $p$. This applies to bond and site percolation on our class of graphs \cs.

The fact that $c(p)<1$ for $p<p_c$ is a celebrated result of Aizenman \& Barsky \cite{AB}. For $p=p_c$ we always have $c(p)=1$. For $p>p_c$ various behaviours can arise depending on the underlying lattice \cite{AiDeSoLow, HerHutSup, KeZhaPro}. Our continuity result applies to the whole interval $p\in (0,1)$.

We will also prove the analogous continuity result for the (upper) exponential growth rate of $\Ex_p(N_n)$, i.e.\ $\limsup_{n\to\infty} \Ex_p(N_n)^{1/n}$, where as before $N_n$ denotes the number of occurring (site-)\pint s. 

\medskip
We will start by proving the continuity of $c(p)$.

\begin{theorem}\label{cp cont}
Consider bond or site percolation on a graph in $\mathcal{S}$. Then $c(p)$ is a continuous function of $p\in (0,1)$.
\end{theorem}
\begin{proof}
Let $I$ be a compact subinterval of $(0,1)$. Define $g_n(p):=\Pr_p(|C_o|=n)^{1/n}$, and notice that $0\leq g_n(p)\leq 1$. Moreover, $g_n$ is a differentiable function with derivative equal to $g_n(p) \dfrac{\Pr_p '(|C_o|=n)}{n\Pr_p(|C_o|=n)}$, where $\Pr_p '(|C_o|=n)$ denotes the derivative of $\Pr_p(|C_o|=n)$. Expressing $\Pr_p '(|C_o|=n)$ via 
$\sum_{P} \big(\dfrac{n}{p}-\dfrac{|\partial P|}{1-p}\big)p^n (1-p)^{|\partial P|}$, where the sum ranges over all lattice (site) animals of size $n$, we conclude that there is a constant $c=c(I)>0$ such that 
$|\Pr_p '(|C_o|=n)|\leq cn \Pr_p (|C_o|=n)$ for every $p\in I$. Therefore, $g_n '$ is uniformly bounded on $I$, hence $g_n$ is a $c$-
Lipschitz function on $I$. The pointwise convergence of $g_n$ to $c(p)$ implies that $c(p)$ is a $c$-Lipschitz function on $I$, hence a continuous function.
\end{proof}

Define $B_p:=\limsup_{n\to\infty} \Ex_p(N_n)^{1/n}$. Before proving the continuity of $B_p$, we will show that $\lim_{n\to\infty} \Ex_p(N_n)^{1/n}$ exists for every $p$.

\begin{proposition}
Consider bond or site percolation on a graph in $\mathcal{S}$. Then for every $p\in (0,1)$, the limit $\lim_{n\to\infty} \Ex_p(N_n)^{1/n}$ exists.
\end{proposition}
\begin{proof}
For simplicity we will prove the assertion for \pint s in $\Z^d$ and $\mathbb{T}^d$. Let $m$ and $n$ be positive integers. We will consider \pint s without any restriction on the surface-to-volume ratio. Arguing as in the proof of \Prr{limit} we combine $m$ \pint s $P_1,P_2,\ldots,P_m$ of size $n$ that have the same shape, and attach a horizontal path to $P_m$, to obtain an \pint\ of size $k=m(n+4)+s$ for some $s$ between $0$ and $n+3$. Notice that the number of attached edges that were initially lying in some $\partial P_i$ is equal to $2m-1$. The probability that the resulting \pint\ occurs is equal to $p^k(1-p)^{M-(2m-1)+N}$, where $M=\sum_{i=1}^m|\partial P_i|$, and $N$ is the number of remaining boundary edges of the \pint . It is not hard to see that $N\leq Cm$ for some constant $C>0$. Hence $$p^k(1-p)^{M-(2m-1)+N}\geq p^{4m+s}(1-p)^{-(2m-1)+Cm} \prod_{i=1}^m p^n (1-p)^{|\partial P_i|}.$$ Summing over all possible sequences $(P_1,P_2,\ldots,P_m)$ we obtain $$\Ex_p(N_k)\geq p^{4m+s} (1-p)^{-(2m-1)+Cm}(\Ex_p(N_n))^m .$$ Taking the $k$th root, and then letting $m\to\infty$ and $n\to \infty$, we obtain that $\liminf_{n\to\infty} \Ex_p(N_n)^{1/n}\geq \limsup_{n\to\infty} \Ex_p(N_n)^{1/n}$, which implies the desired assertion.
\end{proof}

The proof of \Tr{cp cont} applies mutatis mutandis to $B_p$: instead of defining $g_n(p)$ as $\Pr_p(|C_o|=n)^{1/n}$, we define $g_n(p):=\Ex_p(N_n)^{1/n}$, and we use the fact that $\Ex_p(N_n)\leq ln+1$.

\begin{corollary}\label{cor bp cont}
Consider bond or site percolation on a graph in $\mathcal{S}$. Then $B_p$ is a continuous function of $p\in (0,1)$.
\end{corollary}

\bibliographystyle{plain}
\bibliography{collective}

\end{document}